\documentclass[11pt]{amsart}
\usepackage{graphicx}
\usepackage{amsmath,amsthm,amssymb,enumerate}
\usepackage{euscript,mathrsfs}
\usepackage{dsfont}
\usepackage[citecolor=blue,colorlinks=true]{hyperref}
\usepackage[left=3.2cm,right=3.2cm,top=4cm,bottom=4cm]{geometry}
\usepackage{color}
\catcode`\@=11 \@addtoreset{equation}{section}

\catcode`\@=12

\allowdisplaybreaks

\newtheorem{Theorem}{Theorem}[section]
\newtheorem{Proposition}[Theorem]{Proposition}
\newtheorem{Lemma}[Theorem]{Lemma}
\newtheorem{Corollary}[Theorem]{Corollary}

\theoremstyle{definition}
\newtheorem{Definition}[Theorem]{Definition}

\newtheorem{Remark}[Theorem]{Remark}

\newcommand{\bProposition}[1]{
\begin{Proposition} \label{P#1}}
\newcommand{\eP}{\end{Proposition}}

\newcommand{\bLemma}[1]{
\begin{Lemma} \label{L#1} }
\newcommand{\eL}{\end{Lemma}}

\newcommand{\BTm}{\mathfrak{B}(\OTm)}
\newcommand{\BTN}{\mathfrak{B}}
\newcommand{\bTheorem}[1]{
\begin{Theorem} \label{T#1} }
\newcommand{\eT}{\end{Theorem}}
\newcommand{\OTp}{\Omega^{[T, \infty)}}
\newcommand{\OTm}{\Omega^{[0,T]}}
\newcommand{\OTN}{\Omega}

\newcommand{\bCorollary}[1]{
\begin{Corollary} \label{C#1} }
\newcommand{\eC}{\end{Corollary}}

\newcommand{\bRemark}[1]{
\begin{Remark} \label{R#1} }
\newcommand{\eR}{\end{Remark}}

\newcommand{\bDefinition}[1]{
\begin{Definition} \label{D#1} }
\newcommand{\eD}{\end{Definition}}

\newcommand{\dif}{\mathrm{d}}
\newcommand{\Del}{\Delta_x}

\newcommand{\mf}{\mathfrak{F}}
\newcommand{\mB}{\mathfrak{B}}

\newcommand{\prst}{\mathcal{P}}
\newcommand{\p}{\mathcal{P}}
\newcommand{\stred}{\mathbb{E}}

\newcommand{\mn}{\mathbb{N}}

\newcommand{\mt}{\mathbb{T}^N}
\newcommand{\Q}{\mathbb{T}^N}

\newcommand{\bq}{\mathbf q}

\newcommand{\ups}{\partial_t\bfxi^3}

\newcommand{\bfu}{\mathbf{u}}
\newcommand{\bfv}{\mathbf{v}}
\newcommand{\bfV}{\mathbf{V}}
\newcommand{\bfU}{\mathbf{U}}
\newcommand{\bfq}{\mathbf{q}}
\newcommand{\bfy}{\mathbf{y}}
\newcommand{\bfz}{\mathbf{z}}

\newcommand{\bfF}{\mathbf{F}}

\newcommand{\bfG}{\mathbf{G}}

\newcommand{\bfphi}{\boldsymbol{\varphi}}
\newcommand{\bfxi}{\boldsymbol{\xi}}

\newcommand{\bfvarphi}{\boldsymbol{\varphi}}

\newcommand{\bFormula}[1]{
\begin{equation} \label{#1}}
\newcommand{\eF}{\end{equation}}

\newcommand{\PO}{{\rm Prob}[\Omega]}

\newcommand{\CPOa}{\mathrm{Comp}\big({\rm Prob}[\Omega_{X}^{[0,\infty)}]\big)}
\newcommand{\CC}{\mathcal{C}}

\newcommand{\torN}{\mathbb{T}^N}
\newcommand{\8}{\infty}
\newcommand{\diver}{\mathrm{div}}
\newcommand{\pas}{$\mathcal{P}$-a.s.}
\newcommand{\N}{\mathbb{N}}

\newcommand{\ve}{\varepsilon}

\newcommand{\Ov}[1]{\overline{#1}}

\newcommand{\aleq}{\lesssim}

\newcommand{\vr}{\varrho}

\newcommand{\vu}{\vc{u}}

\newcommand{\vc}[1]{{\bf #1}}

\newcommand{\Div}{{\rm div}}
\newcommand{\Grad}{\nabla}

\newcommand{\dd}{{\rm d}}
\newcommand{\dW}{\,{\rm d} {W}}
\newcommand{\dx}{\,{\rm d} {x}}

\newcommand{\dt}{\,{\rm d} t }
\newcommand{\dtau}{\,{\rm d} \tau }

\newcommand{\dxt}{\dx  \dt}

\newcommand{\intO}[1]{\int_{\Omega} #1 \ \dx}
\newcommand{\intQ}[1]{\int_{\torN} #1 \ \dx}

\newcommand{\intTN}[1]{\int_{\torN} #1 \ \dx}

\newcommand{\D}{{\rm d}}
\newcommand{\ep}{\varepsilon}

\newcommand{\R}{\mathbb{R}}

\newcommand{\E}{\mathbb{E}}

\newcommand{\expe}[1]{ \mathbb{E} \left[ #1 \right] }

\newcommand{\OTX}{\Omega^{[T, \infty)}_X}
\newcommand{\ONTX}{\Omega^{[0,T]}_X}
\newcommand{\ONX}{\Omega^{[0, \infty)}_X}

\definecolor{Cgrey}{rgb}{0.85,0.85,0.85}
\definecolor{Cblue}{rgb}{0.50,0.85,0.85}
\definecolor{Cred}{rgb}{1,0,0}
\definecolor{fancy}{rgb}{0.10,0.85,0.10}

\newcommand\Cbox[2]{%
    \newbox\contentbox%
    \newbox\bkgdbox%
    \setbox\contentbox\hbox to \hsize{%
        \vtop{
            \kern\columnsep
            \hbox to \hsize{%
                \kern\columnsep%
                \advance\hsize by -2\columnsep%
                \setlength{\textwidth}{\hsize}%
                \vbox{
                    \parskip=\baselineskip
                    \parindent=0bp
                    #2
                }%
                \kern\columnsep%
            }%
            \kern\columnsep%
        }%
    }%
    \setbox\bkgdbox\vbox{
        \color{#1}
        \hrule width  \wd\contentbox %
               height \ht\contentbox %
               depth  \dp\contentbox
        \color{black}
    }%
    \wd\bkgdbox=0bp%
    \vbox{\hbox to \hsize{\box\bkgdbox\box\contentbox}}%
    \vskip\baselineskip%
}

\date{}

\begin{document}


\title{Markov selection for the stochastic compressible Navier--Stokes system}

\author{Dominic Breit}
\address[D. Breit]{Department of Mathematics, Heriot-Watt University, Riccarton Edinburgh EH14 4AS, UK}
\email{d.breit@hw.ac.uk}

\author{Eduard Feireisl}
\address[E.Feireisl]{Institute of Mathematics AS CR, \v{Z}itn\'a 25, 115 67 Praha 1, Czech Republic}
\email{feireisl@math.cas.cz}
\thanks{The research of E.F. leading to these results has received funding from
the Czech Sciences Foundation (GA\v CR), Grant Agreement
18--05974S. The Institute of Mathematics of the Academy of Sciences of
the Czech Republic is supported by RVO:67985840.}

\author{Martina Hofmanov\'a}
\address[M. Hofmanov\'a]{Fakult\"at f\"ur Mathematik, Universit\"at Bielefeld, D-33501 Bielefeld, Germany}
\email{hofmanova@math.uni-bielefeld.de}

\begin{abstract}
We analyze the Markov property of solutions to the compressible Navier--Stokes system perturbed by a general multiplicative stochastic forcing.
We show the existence of an almost sure Markov selection to the associated martingale problem. Our proof is based on the abstract framework introduced in [F. Flandoli, M. Romito: Markov selections for the 3D stochastic Navier--Stokes
equations. Probab. Theory Relat. Fields 140, 407--458. (2008)]. A major difficulty arises from the fact, different from the incompressible case, that the velocity field is not continuous in time. In addition, it cannot be recovered from the variables whose time evolution is described by the Navier--Stokes system, namely, the density and the momentum. We overcome this issue by introducing an auxiliary variable into the Markov selection procedure.
\end{abstract}


\date{\today}

\maketitle



\section{Introduction}
In this paper we are concerned with the problem of Markov selection for the \emph{compressible Navier--Stokes system} driven by a stochastic forcing
\begin{equation} \label{E1}
\D \vr + \Div (\vr \vu) \dt = 0,
\end{equation}
\begin{equation} \label{E2}
\D (\vr \vu) + \Div (\vr \vu \otimes \vu) \dt + \Grad p(\vr) \dt = \Div \,\mathbb{S}(\Grad \vu) \dt +
 \mathbb{G} (\vr,\vr\vu)\dW,
\end{equation}
\begin{equation} \label{E3}
\mathbb{S}(\Grad \vu) = \mu \left( \Grad \vu + \Grad^t \vu - \frac{2}{N} \Div \vu \mathbb{I} \right)
+ \lambda \Div \vu \mathbb{I},\quad \mu > 0, \ \lambda \geq 0,
\end{equation}
supplemented with space-periodic boundary conditions, that is, the spatial variable $x$ belongs to the flat torus $
\torN = \left( [-1,1]|_{\{ -1, 1 \}} \right)^N$, $N=2,3.$
This system governs the time evolution of density $\vr$ and velocity $\vu$ of a compressible viscous fluid, $p(\vr)$ denotes the pressure and $\mu,\lambda$ are viscosity coefficients. The system is perturbed by a stochastic forcing driven by a cylindrical Wiener process $W$ and a possibly nonlinear dependence on the density $\vr$ and momentum $\vr\vu$, cf. Section \ref{ss:force} for details.
A significant progress has been made recently on the system \eqref{E1}--\eqref{E3} and we refer the reader to the monograph \cite{BFHbook} for a detailed exposition and further references. Here, we would only like to give a brief account of the current state of art, which has led us to writing the present article.

Many fundamental problems in modern continuum mechanics remain largely open and the situation is not different when it comes to the {\em compressible} Navier--Stokes system. In fact, in contrast to the {\em incompressible} counterpart the situation is even more challenging as uniqueness is unknown already in space dimension 2. The only available framework for global existence of \eqref{E1}--\eqref{E3} is the concept of the so-called dissipative martingale solutions established in \cite{BH,BFH2}. These solutions are weak in both PDE and probabilistic sense and in addition they satisfy a suitable version of energy inequality. This way they preserve an important part of information that would be otherwise lost within the construction of ordinary weak solutions. The energy inequality is the cornerstone for further applications and in particular it allows to prove weak--strong uniqueness, see \cite{BFH2}. In \cite{BFH18} it was shown that strong solutions exist locally in time. As these solutions possess sufficient space regularity, they are  unique and as a consequence they are also strong in the probabilistic sense.
Nevertheless, there is still a significant gap in the above theory, namely, the global existence of unique solutions is still missing. The situation is the same in the deterministic setting (see \cite{Li2,FNP}) and, as a matter of fact, also for  the {\em incompressible} Navier--Stokes system in space dimension 3.

An important feature of systems with uniqueness is their memoryless property called Markovianity: Letting the system run from time $0$ to time $s$ and then restarting and letting it run from time $s$ to time $t$ gives the same outcome as letting it run directly from time $0$ to time $t$. In other words, the knowledge of the whole past up to time $s$ provides no more useful information about the outcome at time $t$ than knowing the state of the system at time $s$ only. For systems where the uniqueness is unknown, a natural question is whether there exists a Markov selection. Roughly speaking, for every initial condition the system possesses possibly multiple solutions and each of them generates a probability measure on the space of trajectories, the associated law. Markov selection then chooses one law for every initial condition in such a way that the above explained ``gluing'' property holds. In this way, it is a step in the direction of uniqueness but it shall be stressed that uniqueness  still remains out of reach (see the discussion in \cite{FlaRom,StVa}).

It is worth noting that this approach can be applied also to the standard deterministic fluid model without explicit stochastic terms. The associated measures are then supported on the set 
of all global solutions emanating from given initial data and Markovianity may be interpreted in the same way as above.\\

Existence of a Markov selection for a class of stochastic differential equations has been given by Krylov \cite{Kr}. The crucial observation is that
that Markovianity can be deduced from disintegration property (stability with respect to building conditional expectations) and reconstruction property (stability with respect to ``gluing'' together) of a family of probability laws.
The method has been presented by Stroock--Varadhan \cite{StVa} and generalized to an infinite dimensional setting by Flandoli--Romito \cite{FlaRom} and further by Goldys--R\"ockner--Zhang \cite{GRZ}. Application to a surface growth model has been given by Bl\"omker--Flandoli--Romito \cite{BFR}. In particular, the work by Flandoli--Romito \cite{FlaRom} established the existence of a Markov selection for the 3D incompressible Navier--Stokes system under general additive noise perturbation. In addition, the strong Feller property was shown under stronger assumptions on the noise. Regularity with respect to  initial conditions was proved by Flandoli--Romito \cite{FR2}. Another approach towards existence of Markov solutions and ergodicity for the 3D incompressible Navier--Stokes system based on Galerkin approximations has beed presented by Da~Prato--Debussche \cite{DaPDe1} and Debussche--Odasso \cite{DeOd}.

\medskip

Our paper follows the approach of \cite{FlaRom} and we show the existence of a Markov selection for the system \eqref{E1}--\eqref{E3} (in fact, we have to use the generalization from \cite{GRZ} to Polish spaces due to the complicated structure of the compressible system). Even though the overall structure of the proof is rather similar, we have discovered several interesting challenges along the way. They originate in the significantly more involved structure of the {\em compressible} model \eqref{E1}--\eqref{E3} in comparison to the {\em incompressible} one considered in \cite{FlaRom}. The most striking point with various unpleasant consequences is that \eqref{E1}--\eqref{E3} is a mixed system  whose  solution consists of a couple of density and velocity $[\vr,\vu]$, but the time evolution is only described for density and momentum $[\vr,\vr\vu]$. Furthermore, since the so-called vacuum regions, where the density vanishes, cannot be excluded, it is impossible to gain any information on the time regularity of the velocity. As a consequence, it is only a class of equivalence in time and not a stochastic process in the classical sense.

Therefore, it seems that the natural variables for the desired Markov property is the couple of density and momentum. However, and again due to the presence of the vacuum states, the velocity {\em cannot} be recovered from these variables. In other words, the velocity is {\em not} a measurable function of $[\vr,\vr\vu]$. This fact has already been observed in the proof of existence in \cite{BH}, where the filtration associated to a martingale solution was generated by the density and the velocity. This is  sufficient to recover the momentum $\vr\vu$ as it is a measurable function of $\vr$ and $\vu$. Let us point out that if the equation was deterministic, that is the forcing was of the form $\vr \mathbf{f}\,\dd t$ for some deterministic function $\mathbf{f}$, then (at least under certain boundary conditions) the velocity is a measurable function of $[\vr,\vr\vu]$. In fact, all the terms on the left hand side of the momentum equation \eqref{E2} as well as the forcing can be written as functions of $[\vr,\vr\vu]$ and, as a consequence, also the dissipative term on the right hand side is a function of $[\vr,\vr\vu]$. Nevertheless, under the presence of the stochastic perturbation we can only deduce that the right-hand side of \eqref{E2}, i.e. the sum of the dissipative and the stochastic term, is a measurable function of $[\vr,\vr\vu]$. This is not enough in order to recover the structure of the stochastic integral.

In order to overcome this issue, we introduce an auxiliary variable $\bfU$ together with  an auxiliary equation
$$
\dd\bfU=\bfu \,\dd t,\qquad \bfU(0)=\bfU_{0},
$$
and we establish the existence of a Markov selection for the triple $[\vr,\vr\vu,\bfU]$. Note that this step indeed solves the problem discussed above: since the velocity $\vu$ belongs a.s. to $L^{2}_{\rm{loc}}(0,\infty;W^{1,2}(\mt))$, the new variable $\bfU$ is a continuous stochastic process with trajectories a.s. in $W^{1,2}_{\rm{loc}}(0,\infty;W^{1,2}(\mt))$. In addition, $\bfu$ is a measurable function of $\bfU$ and thus we recover all the quantities in our system from the knowledge of $[\vr,\vr\vu,\bfU]$. Under suitable boundary conditions  we may have alternatively included an auxiliary variable corresponding to the stochastic integral, which would also provide us with the missing piece of information.
Nevertheless, we shall mention that the initial condition $\bfU_{0}$ is rather superfluous. More precisely, for the Markov selection it is necessary to vary the initial condition for the whole triple $[\vr,\vr\vu,\bfU]$ and that is the reason why we included an arbitrary initial condition $\bfU_{0}$. However, for the recovery of $\bfu$, this is not needed and, on the other hand, $\bfU$ is not a function of $\vu$ due to the missing initial datum. Hence the mapping $\bfU\mapsto \vu$ is not injective.

We remark that as an alternative one may establish the existence of a Markov selection for the couple $[\vr,\bfU]$ which would in turn imply the same result for $[\vr,\vr\bfu,\bfU]$ since for a.e. time the momentum can be recovered from $[\vr,\bfU]$. However, this would require a modified definition of a solution to the martingale problem and the proofs would not simplify. Therefore we chose to work directly with the triple $[\vr,\vr\bfu,\bfU]$.

The paper is organized as follows.
In Section \ref{N} we collect some known concepts for probability measures on Polish spaces.
The bulk is the abstract Markov selection in Theorem \ref{thm:2.8}. It is a slight modification of the Markov selection for Polish spaces from \cite{GRZ}. Section
\ref{sec:NS} is concerned with martingale solutions to the compressible Navier--Stokes system \eqref{E1}--\eqref{E3}. We show the equivalence of the concept of dissipative martingale solutions (existence of which has been shown in \cite{BH} and \cite{BFH2}) and a solution to the associated martingale problem. The latter one is a probability law on the space of trajectories, cf. Definition \ref{D} for the precise formulation.
Our main result is contained in Section \ref{s:main}: In Theorem \ref{thm:main} we show the existence of a Markov selection for the system \eqref{E1}--\eqref{E3}.

\section{Probability framework}
\label{N}

Let $X$ be a topological space. The symbol $\mathfrak{B}(X)$ denotes the $\sigma$-algebra of Borel subsets of $X$.
If $\mathcal{U}$ is a Borel measure on $X$, we denote by $\Ov{\mathfrak{B}(X)}$ the $\sigma$-algebra of all Borel subsets of $X$ augmented by all zero measure sets.
The symbol ${\rm Prob}[X]$ denotes the set of all Borel probability measures on a topological space $X$.
In addition,
$
( [0,1], \Ov{\mathfrak{B}[0,1]}, \mathfrak{L} )
$
denotes the \emph{standard probability space}, where $\mathfrak{L}$ is the Lebesgue measure.

\subsection{Trajectory spaces}

Let $(X,d_X)$ be a Polish space. For $T>0$ we introduce the trajectory spaces
\[
\ONTX = C( [0, T]  ; X ),\qquad
\OTX = C_{{\rm loc}}( [T, \infty); X ),\qquad
\ONX = C_{{\rm loc}}( [0, \infty); X ),
\]
and denote $\mathfrak{B}_{T}=\mathfrak{B}(\Omega_{X}^{[0,T]})$.
Note that all the above trajectory spaces are Polish as long as $X$ is Polish.
For $\xi \in \OTX$ we define a time shift,
\[
\mathcal{S}_\tau : \ \OTX \to \Omega^{[T +\tau; \infty)}_X,\
\mathcal{S}_\tau [\xi]_t = \xi_{t - \tau}, \ t \geq T+\tau.
\]
Obviously, the mapping $\mathcal{S}_\tau$ is an isometry.
For a Borel measure $\mathcal V$ on $\OTX$, the time shift
$\mathcal{S}_{{-\tau}}$ is a Borel measure on the space $\Omega^{[T-\tau, \infty)}_X$ given by
\[
\mathcal{S}_{{-\tau}} [\mathcal V] (B) = \mathcal{V} (\mathcal{S}_\tau(B)),\
B \in \mathfrak{B} \big(\Omega^{[T-\tau, \infty)}_X \big).
\]

\subsection{Disintegration}\label{subsecT1}
A conditional probability corresponds to disintegration of a probability measure with respect to a $\sigma$-field.
We report the following result, cf. \cite[Theorem 1.1.6]{StVa}.

\begin{Theorem} \label{T1}

Let $X$ be a Polish space.
Let $\mathcal U \in {\rm Prob}[\ONX]$ and $T \geq 0$.
Then there exists a unique family of probability measures
\[
\mathcal U|^{\tilde \omega }_{\mathfrak{B}_T}  \in {\rm Prob}[\OTX] \
\mbox{for}\ \mathcal U\mbox{-a.a.} \ \tilde \omega
\]
such that the mapping
\[
\ONX \ni\tilde{\omega}  \mapsto \mathcal U|^{\tilde\omega}_{\mathfrak{B}_T} \in {\rm Prob}[\OTX]
\]
is $\mathcal U$-measurable and
the following properties hold:
\begin{enumerate}[(a)]
\item For $\omega \in \OTX$ we have $\mathcal U|^{\tilde\omega}_{\mathfrak{B}_T}$-a.s.
\[
\omega(T) = \tilde{\omega}(T);
\]
\item
For any Borel set ${A} \subset \ONTX$ and any Borel set ${B} \subset \OTX$,
\[
\mathcal U \left( \omega|_{[0,T]} \in {A}, \ \omega|_{[T, \infty)} \in {B} \right) =
\int_{\tilde{\omega} \in {A}} \mathcal U|^{\tilde\omega}_{\mathfrak{B}_T} ( B ) \ {\rm d} \mathcal U(\tilde \omega).
\]
\end{enumerate}
\end{Theorem}

\subsection{Reconstruction}\label{subsecT2}
Reconstruction can be understood as the inverse procedure to disintegration, some sort of ``gluing together'' procedure.
We report the following result, see Lemma
6.1.1 and Theorem 6.1.2 in \cite{StVa}.

\begin{Theorem} \label{T2}

Let $X$ be a Polish space.
Let $\mathcal U \in {\rm Prob}[\ONX]$.
Suppose that $Q_\omega$ is a family of probability measures, such that
\[
\ONX\ni\omega  \mapsto Q_{\omega} \in {\rm Prob}[\OTX],
\]
is $\mathcal U$-measurable.
Then there exists a unique probability measure $\mathcal U \otimes_T Q$ such that
\begin{enumerate}[(a)]
\item For any Borel set ${A} \subset \ONTX$ we have
\[
(\mathcal U \otimes_T Q) (A ) = \mathcal U (A ) ;
\]
\item For $\tilde{\omega} \in \OTN$ we have $\mathcal U$-a.s.
\[
(\mathcal U \otimes_T Q)|^{\tilde \omega}_{\mathcal{B}_T} = Q_{\tilde \omega}.
\]

\end{enumerate}

\end{Theorem}

\subsection{Markov processes}
\label{subsec:markov}

In this subsection we present the abstract framework of almost sure Markov processes as well as the Markov selection theorem.
We  follow the framework of \cite{GRZ} which generalizes the theory from \cite{FlaRom} to Polish spaces. Let $(X,d_X)$ and $(H,d_H)$ be two Polish space,
where the embedding $H\hookrightarrow X$ is continuous and dense. Furthermore, let $Y$ be a Borel subset of $H$. As $(Y,d_H)$ is not necessarily complete and the embedding $Y\hookrightarrow X$ is not assumed to be dense the situation sightly differs form \cite{GRZ}.
A family of probability measures $\{\mathcal U_{y}\}_{y\in Y}$ on $\Omega_{X}^{[0,\infty)}$ is called Markovian if we have for any $y\in Y$ that
\begin{align*}
\mathcal U_{\omega(\tau)}=\mathcal S_{-\tau}  \mathcal U_y|_{\mathfrak B_\tau}^\omega\quad\text{for $\mathcal U_y$-a.a. }\omega\in\Omega_{X}^{[0,\infty)}.
\end{align*}
The following definition is inspired by \cite[Def. 2.3]{GRZ}. It is concerned with probability measures which are supported only on a certain subset of a Polish space.
\begin{Definition}
Let $Y$ be a Borel subset of $H$ and let $\mathcal U\in{\rm Prob}[\Omega_{X}^{[0,\infty)}]$. We say that $\mathcal U$ is concentrated on the paths with values in $Y$ if there is
some $A\in \mathfrak B(\Omega_{X}^{[0,\infty)})$ such that $\mathcal U(A)=1$ and $A\subset \{\omega\in\Omega_{X}^{[0,\infty)}:\,\omega(\tau)\in Y\,\forall \tau\geq0\}$. We write $\mathcal U\in\mathrm{Prob}_Y[\Omega_{X}^{[0,\infty)}]$.
\end{Definition}

The following definition is inspired by \cite[Def. 2.4]{FlaRom} (see also \cite{GRZ} for a version on Polish spaces). It generalizes the classical Markov process
to the situation, where the Markov property only holds for a.e. time-point. It has been introduced for the Navier--Stokes system, where the energy inequality does not hold for all times.
\begin{Definition}[Almost sure Markov property]\label{def:2.4}
Let $y\mapsto \mathcal U_y$ be a measurable map defined on a measurable subset $Y\subset H$ with values in ${\rm Prob}_Y[\Omega_{X}^{[0,\infty)}]$. The family $\{\mathcal U_y\}_{y\in Y}$ has the almost sure Markov property if for each $y\in Y$ there is a set $\mathfrak T\subset(0,\infty)$ with zero Lebesgue measure such that
\begin{align*}
\mathcal U_{\omega(\tau)}=\mathcal S_{-\tau}  \mathcal U_y|_{\mathfrak B_\tau}^\omega\quad\text{for $\mathcal U_y$-a.a. }\omega\in\Omega_{X}^{[0,\infty)}
\end{align*}
for all $\tau\notin \mathfrak T$.
\end{Definition}

The following definition is inspired by \cite[Definition 2.5]{FlaRom} (see also \cite{GRZ} for a version on Polish spaces). It is motivated by the crucial observation by Krylov \cite{Kr}
that Markovianity can be deduced from disintegration and reconstruction of a family of probability laws.
\begin{Definition}[Almost sure pre-Markov family]\label{def:2.5} Let $Y$ be a Borel subset of $H$.
Let $\mathcal C:Y\rightarrow \CPOa\cap \mathrm{Prob}_Y[\Omega_{X}^{[0,\infty)}]$ be a measurable map. The family
$\{\mathcal C(y)\}_{y\in Y}$ is almost surely pre-Markov if for each
$y\in Y$ and $\mathcal U \in \CC(y)$ there is a set $\mathfrak T\subset(0,\infty)$ with zero Lebesgue measure such that the following holds for all $\tau\notin\mathfrak T$
\begin{enumerate}
\item The disintegration property holds, i.e. we have
\begin{align*}
\mathcal S_{-\tau}  \mathcal U|_{\mathfrak B_\tau}^\omega\in \mathcal C(\omega(\tau))\quad\text{for $\mathcal U$-a.a. }\omega\in\Omega_{X}^{[0,\infty)};
\end{align*}
\item The reconstruction property holds, i.e. for each $\mathfrak B_\tau$-measurable map $\omega\mapsto Q_\omega:\Omega_{X}^{[0,\infty)}\rightarrow \mathrm{Prob}(\Omega_{X}^{[\tau,\infty)})$ with
\begin{align*}
\mathcal S_{-\tau} Q_\omega\in \mathcal C(\omega(\tau))\quad\text{for $\mathcal U$-a.a. }\omega\in\Omega_{X}^{[0,\infty)};
\end{align*}
we have $P\otimes_\tau Q\in\mathcal C(y)$.
\end{enumerate}
\end{Definition}

The following theorem states the existence of a Markov selection. It is a slight modification of \cite[Theorem 2.7]{GRZ} which in turn originates from \cite[Theorem 2.8]{FlaRom}.

\begin{Theorem}\label{thm:2.8}
Let $Y$ be a Borel subset of $H$.
Let $\{\mathcal C(y)\}_{y\in Y}$ be an almost sure pre-Markov family (as defined in Definition \ref{def:2.5}) with non-empty convex
values. Then there is a measurable map $y\mapsto \mathcal U_y$ defined on $Y$
with values in ${\rm Prob}_{Y}[\Omega_{X}^{[0,\infty)}]$ such that $\mathcal U_y\in\mathcal C(y)$ for all $y\in Y$ and $\{\mathcal U_y\}_{y\in Y}$ has the almost sure Markov property (as defined in Definition~\ref{def:2.5}).
\end{Theorem}

\begin{proof}
If $Y=H$ the statement is exactly \cite[Thm. 2.7]{GRZ}. We aim to reduce the general situation to this case. Define the map $\tilde{\mathcal{C}}:H\rightarrow\CPOa$ by
\begin{align*}
\tilde{\mathcal{C}}(h):=\begin{cases} \mathcal C(h),\quad h\in Y\\
\{\delta_h\},\quad h\notin Y
\end{cases}.
\end{align*}
Obviously, the map $\tilde{\mathcal{C}}$ has the disintegration and reconstruction property (it is assumed if $y\in Y$ and trivial otherwise). So, we can apply
\cite[Thm. 2.7]{GRZ} to get an almost sure Markov selection $\{\mathcal U_h\}_{h\in H}$ which yields an almost sure Markov selection $\{\mathcal U_y\}_{y\in Y}$ simply by restricting to $Y$.
\end{proof}

\subsection{Almost sure supermartingales}
In this subsection we collect some results on almost sure supermartingales (the supermartingale property only holds for a.a. time-point, see Definition \ref{def:3.2} below) from
\cite{FlaRom}, where $(\Omega,\mathfrak{B},(\mathfrak{B}_{t})_{t\geq0},\mathcal{U})$ denotes a stochastic basis. Almost sure supermartingales  have been invented in \cite{FlaRom} in order to deal with the energy balance of the Navier--Stokes system (which is only known to hold for a.a. time-point).
The following statements are generalizations of well-known statements for supermartingales (see, e.g, \cite{StVa}).

\begin{Definition}[\cite{FlaRom}, Def. 3.2]\label{def:3.2}
Let $\theta$ be an $(\mB_t)$-adapted real-valued stochastic process on $\Omega$.
We call $\theta$ an almost sure $((\mathfrak B_t)_{t\geq 0},\mathcal U)$-supermartingale if we have
\begin{align}\label{eq:def3.2}
\E^{\mathcal U}[\theta_t\mathbf 1_A]\leq \E^\mathcal U[\theta_s\mathbf 1_A]
\end{align}
for a.a. $s\geq0$, all $t\geq s$ and all $A\in\mB_s$. The time-points $s$ for which \eqref{eq:def3.2} holds are called regular times of $\theta$. The time-points $s$ for which \eqref{eq:def3.2} does not hold are called exceptional times of $\theta$.
\end{Definition}

The following two propositions are crucial for the behaviour of almost sure supermartingales when it comes to disintegration and reconstruction of the underlying probability measure.

\begin{Proposition}[\cite{FlaRom}, Prop. B.1]\label{prop:B1}
Let $\theta$ and $\zeta$ be two real-valued continuous and $(\mB_{t})$-adapted stochastic processes on $\Omega$ and let $t_0\geq0$. The following conditions are equivalent.
\begin{enumerate}
\item[(i)] $(\theta_{t})_{t\geq0}$ is a $((\mB_{t})_{t\geq0},\mathcal U)$-square
integrable martingale with quadratic variation $(\zeta_{t})_{t\geq 0}$;
\item[(ii)] For $\mathcal U$-a.a. $\omega\in\Omega$ the stochastic process $(\theta_{t})_{t\geq t_{0}}$ is a $((\mB_{t})_{t\geq t_{0}},\mathcal U|_{\mathfrak B_{t_{0}}}^\omega)$-square integrable martingale wit quadratic variation $(\xi_t)_{t\geq t_0}$ and we have $\E^{\mathcal U}\Big[\E^{\mathcal U|_{\mB_{t_{0}}}^{\cdot}}[\xi_t]\Big]<\infty$ for all $t\geq t_0$.
\end{enumerate}
\end{Proposition}

\begin{Proposition}[\cite{FlaRom}, Prop. B.4]\label{prop:B2}
Let $\alpha$ and $\beta$ be two real-valued adapted processes on $\Omega$ such that $\beta$ is non-decreasing and $\theta=\alpha-\beta$ is left lower semi-continuous.
Let  $t_0\geq0$. The following conditions are equivalent.
\begin{enumerate}
\item[(i)] $(\theta_t)_{t\geq t_0}$ is an almost sure $((\mathfrak B_t)_{t\geq t_{0}},\mathcal U)$-supermartingale and we have
$\E^{\mathcal U}[\alpha_t+\beta_t]<\infty$
for all $t\geq t_0$;
\item[(ii)] For $\mathcal U$-a.a. $\omega\in\Omega$ the process  $(\theta_t)_{t\geq t_0}$ is an almost sure $((\mathfrak B_t)_{t\geq0},\mathcal U|_{\mathfrak B_{t_0}}^\omega)$-supermartingale and we have
$$\E^{\mathcal U|_{\mathfrak B_{t_0}}^\omega}[\alpha_t+\beta_t]<\infty,\quad \E^{\mathcal U}\Big[\E^{\mathcal U|_{\mathfrak B_{t_0}}^\cdot}[\alpha_t+\beta_t]\Big]<\infty,$$
for all $t\geq t_0$.
\end{enumerate}
\end{Proposition}
We finally mention a result which allows to obtain an estimate for the tail-probability of an almost sure supermartingale.
\begin{Proposition}[\cite{FlaRom}, Cor. B.3]\label{cor:B3FR}
Let $\theta$ be a real-valued, left lower semi-continuous and $(\mB_{t})$-adapted processes on $\Omega$. Assume that $(\theta_t)_{t\geq 0}$ is an almost sure $((\mathfrak B_t)_{t\geq0},U)$-supermartingale.
Assume further that we have $\theta_t=\alpha_t-\beta_t$, where $\alpha_{t}$ and $\beta_{t}$ are positive and $(\beta_{t})_{t\geq 0}$ is non-decreasing. Let $a$ be a regular time-point of $\theta$ and $b>a$. Then we have
\begin{align*}
\lambda\, \mathcal U\bigg[\sup_{a\leq t\leq b}\alpha_t\geq\lambda\bigg]\leq \,2\Big(\E^{\mathcal U}\theta_a+\E^{\mathcal U}\lim_{t\nearrow b}\theta_t+\E^{\mathcal U}\beta_b\Big)\quad\forall\lambda>0.
\end{align*}
\end{Proposition}

\section{The compressible Navier--Stokes system}
\label{sec:NS}

In this section we are concerned with martingale solutions the compressible Navier--Stokes system. We present the concept of dissipative martingale solutions living on a complete probability space  $
\left( \mathcal{O}, \mathfrak{F}, (\mathfrak{F}_t)_{t \geq 0}, \mathcal{P} \right)
$
with a complete right-continuous filtration $(\mathfrak{F}_t)_{t \geq 0}$.
Furthermore, we introduce a solution to the martingale problem associated with \eqref{E1}--\eqref{E3} which is a probability law on the space of trajectories.
In Proposition \ref{prop:def=D} we show that both concepts are equivalent.

\subsection{Driving force}
\label{ss:force}

In this subjection we give the precise assumptions on the stochastic forcing
in the momentum equation \eqref{E2}.
The  stochastic process $W$ is a cylindrical $(\mf_t)$-Wiener process in a separable Hilbert space $\mathfrak{U}$. It is formally given by the expansion $W(t)=\sum_{k=1}^\infty e_k\, W_k(t)$ where $(W_k)_{k\in\N}$ is a sequence of mutually independent real-valued  Wiener processes relative to $(\mf_t)_{t\geq0}$ and $(e_k)_{k\in\N}$ is a complete orthonormal system in  $\mathfrak{U}$.
Accordingly, the diffusion coefficient ${\mathbb G}$ is defined as a superposition operator
$\mathbb G(\varrho, \bq):\mathfrak{U}\rightarrow L^1(\torN,R^{N})$,
$${\mathbb G}(\varrho,\bq)e_k=\bfG_k(\cdot,\varrho(\cdot),\bq(\cdot)).$$
The coefficients $\bfG_{k} = \bfG_k (x, \varrho, \vc{q}) :\torN \times[0,\8)\times R^N \rightarrow R^N$ are $C^1$-functions such that there exist constants $(g_k)_{k\in\N}\subset  [0,\infty)$ with $\sum_{k=1}^\8 g_k^2<\8$ and uniformly in $x\in\mt$ it holds
\begin{align}\label{growth1'}
|\bfG_{k}(x,\varrho,\bfq)|&\leq g_k(\varrho+|\bfq|) ,\\
|\nabla_{\varrho,\bfq} \bfG_{k}(x,\varrho,\bfq)|&\leq g_k\label{growth2'}.
\end{align}
Finally, we define the auxiliary space $\mathfrak{U}_0\supset\mathfrak{U}$ via
$$\mathfrak{U}_0=\bigg\{v=\sum_{k\geq1}\alpha_k e_k;\;\sum_{k\geq1}\frac{\alpha_k^2}{k^2}<\infty\bigg\},$$
endowed with the norm
$$\|v\|^2_{\mathfrak{U}_0}=\sum_{k\geq1}\frac{\alpha_k^2}{k^2},\qquad v=\sum_{k\geq1}\alpha_k e_k.$$
Note that the embedding $\mathfrak{U}\hookrightarrow\mathfrak{U}_0$ is Hilbert-Schmidt. Moreover, trajectories of $W$ are $\prst$-a.s. in $C([0,T];\mathfrak{U}_0)$ (see \cite{daprato}).

\subsection{Dissipative martingale solution}

In what follows, we assume that the pressure-density state equation is given by
\[
p(\vr) = a \vr^\gamma, \ a > 0, \ \gamma > \frac{N}{2},
\]
and the corresponding pressure potential reads as
$$
P(\vr)=\frac{a}{\gamma-1}\vr^{\gamma}.
$$
We give a rigorous definition of a solution to \eqref{E1}--\eqref{E3}.

\begin{Definition}[Dissipative martingale solution]\label{def:sol}
The quantity  $\big((\mathcal O,\mf,(\mf_t)_{t\geq0},\prst),\varrho,\bfu,W)$
is called a {\em dissipative martingale solution} to \eqref{E1}--\eqref{E3} provided
\begin{enumerate}[(a)]
\item $(\mathcal O,\mf,(\mf_t)_{t\geq0},\prst)$ is a stochastic basis with a complete right-continuous filtration;
\item $W$ is a cylindrical $(\mf_t)$-Wiener process;
\item the density $\vr\geq0$ belongs to the space
$
C_{\rm{loc}}([0,\infty); (L^{\gamma}(\mt),w))
$
$ \mathcal{P}\mbox{-a.s.}$ and is $(\mf_t)$-adapted;

\item the momentum $\vr\bfu$
belongs to the space
$
 C_{\rm{loc}}([0, \infty); (L^\frac{2\gamma}{\gamma+1}(\mt,R^{N}),w))
$
\pas\ and is $(\mf_t)$-adapted;
\item the velocity $\vu$ belongs to $
 L^2_{\rm{loc}}(0, \infty; W^{1,2}(\mt,R^N))
$
\pas and is $(\mathfrak{F}_{t})$-adapted;
\item the total energy
\[
 E(t) = \int_{\mt}\left[ \frac{1}{2} \frac{|\varrho\bfu(t)|^2}{\vr(t)} +P(\varrho(t)) \right]\dx
\]
belongs to the space
$
L^\infty_{\rm{loc}}(0,\infty)
$
\pas;
\item the equation of continuity
\begin{equation*}
\left[ \int_{\mt} \vr \psi \ \dx \right]_{t = 0}^{t = \tau} -
\int_0^\tau \int_{\torN}\varrho\bfu\cdot\Grad\psi\dxt=0
\end{equation*}
holds for all $\tau > 0$, $\psi\in C^1(\torN)$, $\prst$-a.s.;
\item if $b\in C^1(\R)$ such that there exists $M_{b}>0$ with $b'(z)=0$ for all $z\geq M_b$, then
\begin{equation*}
\begin{split}
&\left[ \int_{\mt} b(\vr) \psi \ \dx \right]_{t = 0}^{t = \tau}\\
 &- \int_0^\tau \int_{\torN} b(\varrho)\bfu\cdot\Grad\psi \,\dif x\,\dif t+
\int_0^\tau \int_{\torN} \big(b'(\varrho)\varrho-b(\varrho)\big)\diver\bfu\,\psi \,\dif x\,\dif t=0.
\end{split}
\end{equation*}
for all $\tau > 0$, $\psi \in C^1(\torN)$, $\prst$-a.s.;
\item the momentum equation
\begin{align}
\nonumber
 \left[ \intTN{ \vr \vu \cdot \bfphi } \right]_{t = 0}^{t= \tau}
 &-\int_0^\tau \intTN{ \Big[ \vr \vu \otimes \vu : \Grad \bfphi  +  p(\vr) \Div \bfphi  \Big]  }
\dt\\
&+\int_0^\tau \  \intTN{\mathbb{S}(\Grad \vu) : \Grad \bfphi  } \dt\nonumber\\&= \sum_{k=1}^\infty\int_0^\tau  \left( \intTN{  {\vc{G}_k} (\vr, \vr\vu) \cdot \bfphi } \right) \, \D  W_k
\label{N2}
\end{align}
holds for all $\tau > 0$,  $\bfvarphi\in C^1(\torN;R^N)$, $\prst$-a.s.;
\item
the energy inequality
\begin{equation} \label{N3}
\begin{split}
\frac{1}{n} \left[ \mathbb{E} \Big[ \mathbf 1_{\mathfrak{U}} {E}^n \Big] \right]_{t = \tau_1}^{t = \tau_2}
&+ \mathbb{E} \left[ \mathbf 1_{\mathfrak{U}} \int_{\tau_1}^{\tau_2}  {E}^{n-1} \intTN{ \mathbb{S}(\Grad \vu) : \Grad \vu }
\dt \right] \\
&\leq \mathbb{E} \left[ \mathbf 1_{\mathfrak{U}} \int_{\tau_1}^{\tau_2} {E}^{n-1} \sum_{k = 1}^\infty \intTN{
\vr^{-1} |\mathbf{G}_k (\vr, \vr \vu) |^2 } \dt                           \right] \\
&+ \frac{n-1}{2} \mathbb{E} \left[ \mathbf 1_{\mathfrak{U}} \int_{\tau_1}^{\tau_2} {E}^{n-2} \sum_{k = 1}^\infty \left( \intTN{
 \mathbf{G}_k (\vr, \vr \vu)  } \right)^2 \dt            \right]
\end{split}
\end{equation}
holds for any $n = 0,1,\dots$, any $\tau_2 \geq 0$ and a.a. $\tau_1$, $0 \leq \tau_1 \leq \tau_2$, including $\tau_1 = 0$, and any
$\mathfrak{U} \in \mathfrak{F}_{\tau_1}$.
\end{enumerate}
\end{Definition}

\begin{Remark} \label{ER2}

It is worth noting that it is enough to require validity of the integral identities (g)--(i)
for a countable family of test function that may be formed by the trigonometric polynomials.

\end{Remark}

Note that unlike the density $\vr$ and the momentum $\vr\vu$, the velocity field $\vu$ is not a stochastic process in the classical sense as it is only defined for a.a. time. Thus, adaptedness of $\vu$ to the filtration $(\mathfrak{F}_t)_{t \geq 0}$ shall be understood  in the sense of random distributions introduced in \cite[Section 2.2]{BFHbook}. Namely, the random variable
\[
\int_0^\infty \intTN{ \vu \cdot \bfphi } \dt
\]
is $\mathfrak{F}_\tau$ measurable whenever $\bfphi \in C^\infty_c([0, \tau) \times \torN, R^N)$. This can be reformulated
by means of the following observation.

\begin{Lemma} \label{LE1}

Let
$
\left(\Omega, \mathfrak{F}, (\mathfrak{F}_t)_{t \geq 0}, \mathcal{P} \right)
$
be a stochastic basis.
Then the following statements are equivalent:

\begin{enumerate} [(a)]

\item

$\vu$ is an $(\mathfrak{F}_t)$-adapted  random distribution taking values in $L^2_{{\rm loc}}(0,\infty; W^{1,2}(\torN, R^N))$ $\mathcal{P}$-a.s.;

\item the stochastic process
\[
\vc{U}:\
t \mapsto \int_0^t \vu (s,\cdot) \, {\rm d}s \in W^{1,2}(\torN, R^N)
\]
is $(\mathfrak{F}_t)$-adapted and takes values in $W^{1,2}_{\rm{loc}}(0,\infty;W^{1,2}(\mt,R^{N}))$ $\mathcal{P}$-a.s.
\end{enumerate}

\end{Lemma}

\begin{proof}

The implication (a) $\Rightarrow$ (b) is obvious. To show (a) $\Rightarrow$ (b) we observe that
\begin{equation*}
\int_0^\infty \intTN{ \vu \cdot \bfphi } \dt = - \int_0^\infty \intTN{ \vc{U} \cdot \partial_t \bfphi } \dt
\end{equation*}
for any $\bfphi\in C^{\infty}_{c}((0,\infty)\times\mt,R^{N})$
whence the desired conclusion follows from adaptedness of $\vc{U}$.
\end{proof}

We have the following existence result.

\begin{Theorem} \label{thm:existence}

Let $k>\tfrac{N}{2}$ and
let $\Lambda $ be a Borel probability measure defined on the space $W^{-k,2}(\Q) \times W^{-k,2}(\Q,R^N)$ such that
\begin{align*}
\Lambda&\big\{L^1(\Q) \times L^1(\Q,R^N)\big\}=1,\
\Lambda \{ \vr \geq 0 \} = 1, \\ &\quad\Lambda \bigg\{ 0 < \underline{\vr} \leq \intQ{ \vr } \leq \Ov{\vr} < \infty \bigg\} = 1,
\end{align*}
for some deterministic constants $\underline{\vr}$, $\Ov{\vr}$, and
\[
\int_{L^1_x \times L^1_x}  \left|\, \intQ{ \left[ \frac{1}{2} \frac{|\vc{q} |^2}{\vr} + P(\vr) \right] }
\right|^n   \D \Lambda \textcolor{red}{\leq c(n)}
\]
for $n=1,2,\dots$. Let the diffusion coefficients $\mathbb{G} = (\mathbf{G}_{k})_{k\in\N}$ be continuously differentiable
satisfying \eqref{growth1'} and \eqref{growth2'}.
Then there is a dissipative martingale solution to \eqref{E1}--\eqref{E3} in the sense of Definition \ref{def:sol} with $\Lambda=\mathcal{L}[\vr(0),\vr\bfu(0)]$. \end{Theorem}

\begin{proof}
Theorem \ref{thm:existence} is only a variant of \cite[Thm. 4.0.2.]{BFHbook}.
The proof is based on a four layer approximation scheme where on each layer the stochastic compactness method based on the Jakubowski--Skorokhod representation theorem \cite{jakubow} is used. Since the formulation of the energy inequality \eqref{N3} is slightly different from that in \cite{BFHbook}, we discuss the main points of the proof in the sequel.

We consider a suitable approximation of the diffusion coefficients. It is convenient to introduce $\mathbb F=\big(\vc{F}_k\big)_{k\in\N}$ by
 $$\vc{F}_k(\varrho,\bfu)=\frac{\vc{G}_k(\varrho,\varrho\bfu)}{\varrho}.$$
Note that, in accordance with hypotheses \eqref{growth1'}--\eqref{growth2'}, the functions $\vc{F}_k$
satisfy the following
\[
\vc{F}_k : {\Q} \times [0, \infty) \times R^N \to R^N ,\ \vc{F}_k \in C^1({\Q} \times (0, \infty) \times R^N),
\]
and there exist  constants $(f_k)_{k\in\N}\subset  [0,\infty)$ such that
\begin{equation*}
\| \vc{F}_k (\cdot, \cdot, 0)  \|_{L^\infty_{x,\varrho}} + \| \nabla_{ \vc{u}} \vc{F}_k \|_{L^\infty_{x,\varrho,\bfu}} \leq f_k,\quad\ \sum_{k=1}^\infty f_k^2 < \infty.
\end{equation*}
Finally, we introduce a regularized noise coefficient
 $\mathbb{F}_\varepsilon=\big(\vc{F}_{k,\ep}\big)_{k\in\N}$ by cutting off small values of the density and large values of the velocity.
The basic \emph{approximate problem} then reads as
\begin{align*} 
&\D \vr + \chi(\| \vu \|_{H_m} - R )\Div (\vr [\vu]_{R} ) \dt  = \ep \Del \vr  \dt ,\\
\D \Pi_m [\vr \vu]
&+ \Pi_m [\chi(\| \vu \|_{H_m} - R )\Div (\vr \vu \otimes \vu) ] \dt
+  \Pi_m \big[ \chi(\| \vu \|_{H_m} - R )\Grad p_\delta (\vr)  \big] \dt\nonumber\\
&= \Pi_m \big[ \ep \Del (\vr \vu) +  \Div \mathbb{S}(\Grad \vu) \big] \dt +
\Pi_m \left[ \vr \Pi_m[ \mathbb{F}_\ep(\vr, \vc{u} ) ]  \right]  \dd W,
\end{align*}
where we recognize the artificial viscosity terms $\ep \Del \vr$, $\ep \Del (\vr \vu)$, pressure
regularization $\delta(\vr+ \vr^\Gamma)$ as well as the cut-off operators applied to various quantities using the function
\[
\chi \in C^\infty(\R), \qquad \chi(z) = \left\{ \begin{array}{l} 1 \ \mbox{for} \ z \leq 0, \\ \chi'(z) \leq 0 \ \mbox{for}\ 0 < z < 1, \\
\chi(z) = 0 \ \mbox{for}\ z \geq 1, \end{array} \right.
\]
together with the operators
\[
[\vc{v}]_R = \chi(\| \vc{v} \|_{H_m} - R ) \vc{v}, \ \mbox{defined for}\  \vc{v} \in H_m,\ R\in\N,
\]
where $H_{m}$ is a finite dimensional function space of dimension $m$.
Finally, $\Pi_m$ is a projection operator onto  $H_{m}$. The aim is to pass to the limits $R\to\infty$, $m\to\infty$, $\varepsilon\to 0$ and $\delta\to0$ (in this order) using the stochastic compactness method.

There are now two principal differences to \cite{BFHbook}:
 namely, we are dealing with an infinite time-interval and  the energy inequality in \cite[Thm. 4.0.2.]{BFHbook}
is only included for $n=1$.  The first issue only requires a fine tuning
of the stochastic compactness argument similar to \cite[Sec. 4]{BFHM}: If $X$ is a reflexive separable Banach space and $q\in(1,\infty)$ then topological spaces of the form
\begin{align*}
L^q_{\text{loc}}([0,\infty);X),\quad (L^q_{\text{loc}}([0,\infty);X),w),\quad C_{\text{loc}}([0,\infty);(X,w)),
\end{align*}
belong to the class of the so-called sub-Polish spaces. That is, there exists a countable family of continuous functions that separate points (see \cite[Definition 2.1.3.]{BFHbook}).
Indeed, $L^q_{\text{loc}}([0,\infty);X)$ is a separable metric space with the metric given by
\begin{align*}
(f,g)\mapsto \sum_{M\in\mn}2^{-M}\big(\|f-g\|_{L^q(0,M;X)}\wedge 1\big).
\end{align*}
A set $\mathcal{K}\subset L^q_{\text{loc}}([0,\infty);X)$ is compact provided  the sets
$$\mathcal{K}_M:=\{f|_{[0,M]};\,f\in\mathcal{K}\}\subset L^q(0,M;X)$$
are compact for every $M\in\mn$. On the other hand, the remaining two spaces are  (generally) nonmetrizable locally convex topological vector spaces, generated by the seminorms
$$
f\mapsto \int_0^M \langle f(t), g(t) \rangle_X\dt,\quad M\in\mn,\,g\in L^{q'}(0,\infty; X^*),\,\tfrac{1}{q}+\tfrac{1}{q'}=1,
$$
and
$$
f\mapsto \sup_{t\in[0,M]}\langle f(t), g \rangle_X,\qquad M\in\mn,\,g\in  X^*,
$$
respectively. As above, a set $\mathcal{K}$ is compact provided its restriction to each interval $[0,M]$ is compact in $(L^q(0,M;X),w)$ and  $C([0,M];(X,w))$, respectively. So, in the spaces above there exists a countable family of continuous functions that separate points.
Consequently, the Jakubowski--Skorokhod theorem \cite[Theorem 2]{jakubow} applies.

Let us now discuss the  energy inequality \eqref{N3}. On the basic level (with $R,m,\ve$ and $\delta$ fixed), and in fact even after passing with $R\rightarrow\infty$, we are dealing with finite dimensional function spaces. Hence, the classical version of It\^{o}'s formula applies and we obtain the following energy balance
 arguing similarly to \cite[Prop. 4.1.14.]{BFHbook}
\begin{align}
 &-\frac{1}{n}\int_0^\infty \partial_t \phi  E_{\delta}^n \dt+ \int_0^\infty \phi  E_{\delta}^{ n - 1}\int_{\Q} \left[ \mathbb{S}(\Grad \vu) : \Grad \vu  + \ep  \vr |\Grad \vu|^2 + \ep  P_\delta''(\vr) |\Grad \vr|^2 \right] \dx \dt  \nonumber\\
&\qquad
 = \frac{1}{2}\sum_{k=1}^\8 \int_0^\infty \phi E_{\delta}^{ n - 1} \int_{\Q} \vr |\Pi_m [\vc{F}_{k,\ep} (\vr ,  \vu )] |^2 \dx \dt \nonumber\\
&\qquad+
\sum_{k=1}^\infty\int_0^\infty\phi E_{\delta}^{ n - 1} \int_{\Q}  \vr \Pi_m [ \bfF_{k,\ep} (\vr ,  \vu ) ]  \cdot \vu \dx  \,\dd W_{k},\nonumber\\
&\qquad+ \frac{n-1}{2} \sum_{k=1}^\infty\int_0^\infty \phi  E_{\delta}^{ n - 2 }
 \left( \intQ{    \vr \Pi_m [ \bfF_{k,\ep} (\vr ,  \vu ) ]\cdot \vu } \right)^2  \dt
+\frac{\phi(0)}{n}E^{n}_{\delta}(0).\label{wWS122}
\end{align}
It holds for all $\phi \in C_c^\infty([0,\infty))$ \pas\  with the approximate pressure potential
\[
P_\delta(\vr) = \vr \int_1^\vr \frac{p_\delta(z)}{z^2} \ {\rm d}z=P(\vr)+ \delta \left( \vr \log(\vr) + \frac{1}{\Gamma - 1} \vr^\Gamma \right)
\]
and the total energy
$$
E_{\delta}(t)=\int_{\Q}\left[\frac12\frac{|\vr\bfu(t)|^{2}}{\vr(t)}+P_{\delta}(\vr(t))\right]\dx.
$$
From \eqref{wWS122} one can deduce the moment estimates
\begin{align}
&\expe{  \sup_{\tau \in [0, T]} E_{\delta}^{nr}(\tau)
 }+ \expe{ \left| \int_0^T E_{\delta}^{ n - 1}\int_{\Q} \left[ \mathbb{S}(\Grad \vu) : \Grad \vu  + \ep  \vr |\Grad \vu|^2 + \ep   P_\delta''(\vr) |\Grad \vr|^2 \right] \dx \dt \right|^r }\nonumber\\
  &\qquad\leq c(T)\,
\big(1+\expe{ E_{\delta}^{nr}(0)}) \ \mbox{whenever}\ \ r \geq 2\label{wWS27}
\end{align}
as in \cite[Prop. 4.2.3.]{BFHbook} for all $n\in\N$. The moment bounds from \eqref{wWS27}
can be used to show tightness of the probability laws. Eventually, on uses the Jakubowski--Skorokhod theorem to obtain compactness on a new probability space.
Thanks to \cite[Thm. 2.9.1.]{BFHbook} the energy balance \eqref{wWS122}
continues to hold on the new probability space. The passage to the limit $m\rightarrow\infty$ in \eqref{wWS122} can still be done along the lines of \cite[Lemma 4.3.16.]{BFHbook}. It follows from the passage to the limit in the stochastic integral (see \cite[Prop. 4.3.14.]{BFHbook}) and the arbitrary high moment estimates \eqref{wWS27}.
The subsequent limits $\ve\rightarrow0$
and $\delta\rightarrow0$ follow along the lines of \cite{BFHbook} with the same modifications.  Only the energy inequality \eqref{N3} needs some further explanation (where we follow \cite{FlaRom}, proof of Lemma A.3). So far, we have only shown that for any $\tau>0$ there is a nullset
$\mathfrak T_\tau$ such that\footnote{In \eqref{wWS122}, approximate $\chi_{[r,t]}$ by a sequence of smooth functions $\phi_m$, multiply by $\mathbf 1_{\mathfrak U}$ and apply expectations. In the limit procedures $m\rightarrow\infty$, $\ve\rightarrow0$ and $\delta\rightarrow0$ we use lower semi-continuity on the left-hand-side for any time and on the right-hand side strong convergence for a.a. time.}
\begin{align}
\label{N3b}
\E\big[\mathbf 1_{\mathfrak U}\mathfrak S^n[\varrho,\bfu]_{\tau} \big]\leq\E\big[\mathbf 1_{\mathfrak U}\mathfrak S^n[\varrho,\bfu]_{r} \big]
\end{align}
for all $r\notin\mathfrak T_\tau$ and all $\mathfrak U\in\mf_{r}$, where
\[
\begin{split}
\mathfrak S^n[\vr,\bfu]_\tau= & \frac{1}{n} {E}_\tau^n + \int_0^\tau \left( {E}^{n-1}_t \intQ{ \mathbb{S}(\Grad \bfu)
: \Grad \bfu } \right) \dt\\& - \frac{1}{2} \int_0^\tau \left( {E}^{n-1}_t \sum_{k=1}^\infty \intQ{ \frac{ |\vc{G}_k(\vr,\vr\bfu)|^2}{\vr} }
\right) \dt \\
&- \frac{n-1}{2} \int_0^\tau \left( {E}^{n-2}_t \sum_{k=1}^\infty \left( \intQ{ \vc{G}_k(\vr,\vr\bfu)\cdot \bfu } \right)^2
\right) \dt ,\\
E_t^n&=\int_{\mt}\big[\varrho(t)|\bfu(t)|^2+P(\varrho(t))\big]\dx.
\end{split}
\]
Now we set $\mathfrak T=\bigcup_{t\in\mathfrak D}$ where $\mathfrak D\subset [0,\infty)$ is countable and dense. We claim that
\eqref{N3b} holds for all $r\not\in\mathfrak T$ and all $\tau>r$ which gives \eqref{N3}. In fact, for $0<r<\tau$ with $r\not\in\mathfrak T$ there is a sequence
$(\tau_m)\subset\mathfrak D$ with $\tau_m\rightarrow\tau$. Now, passing with $m\rightarrow\infty$ in \eqref{N3b} and using lower semi-continuity of the mapping
$$t\mapsto \intQ{ \left[ \frac{1}{2}  \frac{|\vr\vu(t)|^2}{\vr(t)} + P(\vr(t)) \right] }$$
yields \eqref{N3}.
\end{proof}
\begin{Remark}\label{rem:moments}
\begin{itemize}
\item[(a)] It can be seen from the proof of Theorem \ref{thm:existence} that is possible to show a much stronger version of the energy inequality which reads as
\begin{align*} \nonumber
- &\frac{1}{n} \int_0^\infty \partial_t \phi E^n \dt -\frac{1}{n} \phi(0) E^{n}(0)
+ \int_0^\infty \phi  E^{ n - 1}
\intQ{ \mathbb{S}(\Grad \vu) : \Grad \vu }  \dt \nonumber
\\
& \leq \frac{1}{2} \int_0^\infty \phi  E^{ n - 1}
\sum_{k=1}^\infty \intQ{ \vr^{-1} |\vc{G}_k(\vr,\vr\bfu)|^2 }  \dt \\
&+ \frac{n-1}{2} \int_0^\infty \phi  E^{ n - 2 }
\sum_{k=1}^\infty \left( \intQ{  \vc{G}_k (\vr,\vr\bfu)\cdot \vu } \right)^2  \dt\nonumber
\\
&+ \sum_{k = 1}^\infty\int_0^\infty \phi  E^{ n - 1}
 \intQ{ \vc{G}_k(\vr,\vr\bfu) \cdot \vu }\, {\rm d}W_k
\nonumber
\end{align*}
for all $\phi\in C^\infty_c([0, \infty))$, $\phi \geq 0$ and all $n=1,2,\dots$ \pas \ The reason why we decided for \eqref{N3} is that otherwise we are unable to show the equivalence of Definition \ref{def:sol} and Definition \ref{D} (see Proposition \ref{prop:def=D}).
\item[(b)] The energy inequality \eqref{N3} and Proposition \ref{cor:B3FR} imply
\begin{align*}
\p\bigg(\sup_{0\leq \tau\leq T}\bigg[\int_{\mt}\frac{|\vr\bfu(\tau)|^{2}}{\vr(\tau)}+P(\varrho(\tau))\big]\dx\bigg]^n+\bigg[\int_0^T\int_{\mt}|\nabla\bfu|^2\dxt\bigg]^n<\infty\bigg)=1
\end{align*}
for all $n\in\N$ and all $T>0$ provided we have
\begin{align*}
\E\bigg[\int_{\mt}\bigg[\frac{|\vr\bfu(0)|^2}{\varrho(0)}+P(\varrho(0))\bigg]\dx\bigg]^n<\infty.
\end{align*}
\end{itemize}
\end{Remark}

\begin{Remark} \label{R1}

In view of the Skorokhod representation theorem, we may always assume that $(\mathcal{O},\mathfrak{F},\mathcal{P})$ is the standard probability space with
$\mathcal{P}$ being the Lebesgue measure on $[0,1]$.

\end{Remark}

\subsection{Martingale solutions as measures on the space of trajectories}

As it can be seen in the proof of Theorem \ref{thm:existence}, the natural filtration associated to a dissipative martingale solution in the sense of Definition \ref{def:sol} is the joint canonical filtration of $[\vr,\vu,W]$. Note that since we cannot exclude vacuum regions where the density vanishes, this filtration differs from the filtration generated by $[\vr,\vr\vu,W]$. In other words the velocity $\bfu$ is not a measurable function of the density and momentum $[\vr,\vr\vu]$. However, as already mentioned above, the velocity is a class of equivalence with respect to all the variables $\omega,t,x$ and is therefore not a stochastic process in the classical sense. Consequently, it is not clear at first sight, how Markovianity for the system \eqref{E1}--\eqref{E3} shall be formulated.

In order to overcome this issue, we introduce a new variable $\bfU$ which corresponds to the time integral $\int_{0}^{\cdot}\bfu\,\dd s$ and we study the Markov selection for the joint law of $[\vr,\vr\vu,\bfU]$. This stochastic process has continuous trajectories and contains all the necessary information. In particular, the velocity $\bfu$ is a measurable function of $\bfU$. However, as the initial condition for $\bfU$ is changing through the proof of the Markov selection (more precisely, we have $\bfU_{t}=\bfU_{0}+\int_{0}^{t}\vu\,\dd s$), the mapping $\bfU\mapsto \vu$ is not injective.

For future analysis, it is more convenient to consider  martingale solutions as probability measures
$U\in {\rm Prob}[\OTN]$,
where
\[
\OTN = C_{\rm loc}([0, \infty);W^{-k,2}(\mt,R^{2N+1})),
\]
where  $k>\tfrac{N}{2}$. This refers is $X=W^{-k,2}(\mt,R^{2N+1}))$ in the set-up of Section \ref{subsec:markov}.
To this end, let $\mathfrak{B}$ denote the Borel $\sigma$-field on $\Omega$.
Let $\bfxi=(\xi^{1},\bfxi^{2},\bfxi^{3})$ denote the canonical process of  projections, that is,
$$
\bfxi=(\xi^{1},\bfxi^{2},\bfxi^{3}):\Omega\to\Omega,\qquad \bfxi_{t}=(\xi^{1}_{t},\bfxi^{2}_{t},\bfxi^{3}_{t})(\omega)=\omega_{t}\in W^{-k,2}(\mt,R^{2N+1}) \ \mbox{for any}\ t \geq 0,
$$
and let $(\mathfrak{B}_{t})_{t\geq 0}$ denote the associated canonical filtration given by
$$
\mathfrak{B}_{t}:=\sigma(\bfxi|_{[0,t]}),\qquad t\geq 0,
$$
which coincides with the Borel $\sigma$-field on $\Omega^{[0,t]}=C([0,t];W^{-k,2}(\mt,R^{2N+1}))$.

To a dissipative martingale solution $\big((\mathcal O,\mf,(\mf_t)_{t\geq0},\prst),\varrho,\bfu,W)$ in the sense of Definition~\ref{def:sol} we may associate its probability law
\[
U = \mathcal{L}\left[\vr, \vc{q} = \vr \bfu,\bfU=\int_0^\cdot\bfu\,\dd s\right]\in {\rm Prob}[\OTN].
\]
We obtain a probability space $\big( \OTN, \BTN, U \big)$.
Finally, we introduce the space
\begin{align*}
H &= \left\{ [\vr, \vc{q},\bfU] \in\tilde H  \ \Big|  \intQ{  \frac{ |\vc{q}|^2}{|\vr|}
 } <\infty  \right\},\\
\tilde H &= L^\gamma (\mt) \times L^{\frac{2 \gamma}{\gamma + 1}}(\mt, R^N)\times W^{1,2}(\mt,R^N).
\end{align*}
We tacitly include points of the form $(0,\vc{0},\bfU)$ with $\bfU\in W^{1,2}(\mt;R^N)$ in $H$. Hence
it is a Polish space together with the metric
\begin{align}\label{eq:dy}
d_H(y,z)&=d_Y((y^1,\bfy^2,\bfy^3),(z^1,\bfz^2,\bfz^3))=\|y-z\|_X+\Big\|\frac{\bfy^2}{\sqrt{|y^1|}}-\frac{\bfz^2}{\sqrt{|z^1|}}\Big\|_{L^2_x}.
\end{align}

Moreover, it is easy to see that the inclusion $H\hookrightarrow X$ is dense.
We also define the subset
\[
Y = \left\{ [\vr, \vc{q},\bfU] \in X \ \Big| \vr\not\equiv 0,\ \vr \geq 0, \ \intQ{  \frac{ |\vc{q}|^2}{\vr}
 } <\infty  \right\}.
\]
Note that $(Y,d_H)$ is not complete (because of $\vr\not\equiv 0$)
and the inclusion $Y\hookrightarrow X$ is not dense (because of $\vr\geq0$).

The law $U(t, \cdot)$ will be supported on $Y$ which consequently also determines the set of admissible initial conditions. This is a consequence of the energy inequality (recall Remark~\ref{rem:moments} (b)) and the continuity equation (which excludes trivial density states by the balance of mass). The following is a rigorous definition.

\begin{Definition}
\label{D}
A Borel probability measure $U$ on $\OTN$ is called a solution to the martingale problem associated to
 \eqref{E1}--\eqref{E3} provided
\begin{enumerate}[(a)]
\item it holds
\begin{align*}
U& \left( \xi^1 \in C_{\rm loc}\big([0, \infty); \big(L^{\gamma}(\mt),w\big)\big),\ \xi^1\geq0 \right)  = 1,\\
U &\left( \bfxi^2 \in C_{\rm loc}\big([0, \infty); \big(L^{\frac{2\gamma}{\gamma+1}}(\mt,R^N),w\big)\big) \right)  = 1,\\
U &\left( \bfxi^3 \in W^{1,2}_{\rm loc}\big([0, \infty); W^{1,2}(\mt,R^N)\big) \right)  = 1;
\end{align*}
\item it holds
$
\bfxi^2 = \xi^1 \partial_t\bfxi^3
$  $U$-a.s.;
\item the total energy
\[
\mathcal E_{t} = \int_{\mt}\left[ \frac{1}{2} \frac{|\bfxi_{t}^2|^2}{\xi_{t}^1} +P(\xi_{t}^1) \right]\dx
\]
belongs to the space
$
L^\infty_{\rm{loc}}(0,\infty)
$
$U\mbox{-a.s.}$;
\item
it holds $U$-a.s.
\[
\left[ \intQ{ \xi_{t}^1 \psi } \right]_{t = 0}^{t = \tau} -
\int_0^\tau \intQ{ \bfxi_{t}^2 \cdot \Grad \psi  } \dt = 0
\]
for any $\psi \in C^1(\mt)$ and  $\tau \geq 0$;
\item  if $b\in C^1(\R)$ such that there exists $M_{b}>0$ with $b'(z)=0$ for all $z\geq M_b$, then
 there holds
$U$-a.s.
\[
\begin{split}
&\left[ \intQ{ b(\xi_{t}^1) \psi } \right]_{t = 0}^{t = \tau} \\
&- \int_0^\tau \intQ{ \left[ b(\xi_{t}^1) \ups_{t} \cdot \Grad \psi + \left( b(\xi_{t}^1) - b'(\xi_{t}^1) \xi_{t}^1 \right)
\Div \ups_{t}\, \psi \right] } \dt = 0
\end{split}
\]
for any  $\psi\in C^1(\torN)$ and  $\tau \geq 0$.
\item for any $\bfphi \in C^1(\mt, R^N)$,
the stochastic process
\[
\begin{split}
&\mathscr M(\bfphi): [\omega, \tau] \mapsto \left[ \intQ{ \bfxi_{t}^2 \cdot \bfphi } \right]_{t = 0}^{t = \tau} -
\int_0^\tau \intQ{ \left[ \frac{\bfxi_{t}^2 \otimes \bfxi_{t}^2}{\xi_{t}^1} : \Grad \bfphi
+ p(\xi_{t}^1) \Div \bfphi \right] } \dt\\
 &\qquad\qquad\qquad\qquad+\int_0^\tau \intQ{ \mathbb{S}(\Grad \ups_{t}) : \Grad \bfphi }\dt
\end{split}
\]
is a square integrable $((\mathfrak{B}_{t})_{t\geq0},U)$-martingale with quadratic variation
\[
\frac{1}{2} \int_0^\tau \sum_{k=1}^\infty \left( \intQ{ \vc{G}_k(\xi_{t}^1,\bfxi_{t}^2) \cdot \bfphi } \right)^2\dt;
\]
\item for any $n=1,2,\dots$
the stochastic process
\[
\begin{split}
\mathscr S^n:[\omega, \tau] &\mapsto \frac{1}{n} \mathcal{E}_\tau^n + \int_0^\tau \left( \mathcal{E}^{n-1}_t \intQ{ \mathbb{S}(\Grad \ups_{t})
: \Grad \ups_{t} } \right) \dt\\& - \frac{1}{2} \int_0^\tau \left( \mathcal{E}^{n-1}_t \sum_{k=1}^\infty \intQ{ \frac{ |\vc{G}_k(\xi_{t}^1,\bfxi_{t}^2)|^2}{\xi_{t}^1} }
\right) \dt \\
&- \frac{n-1}{2} \int_0^\tau \left( \mathcal{E}^{n-2}_t \sum_{k=1}^\infty \left( \intQ{ \vc{G}_k(\xi_{t}^1,\bfxi_{t}^2)\cdot \ups_{t} } \right)^2
\right) \dt
\end{split}
\]
is an almost sure $((\mathfrak{B}_{t})_{t\geq0},U)$-supermartingale (in the sense of Definition \ref{def:3.2}) and $s=0$ is a regular time.
\end{enumerate}
\end{Definition}

The relation between Definition \ref{def:sol} and Definition \ref{D} is given by the following result.

\begin{Proposition}\label{prop:def=D}
The following statements hold true:
\begin{enumerate}
\item Let $((\mathcal O,\mf,(\mf_t)_{t\geq0},\prst),\varrho,\bfu,W)$ be a dissipative martingale solution to \eqref{E1}--\eqref{E3} in the sense of Definition \ref{def:sol}. Then for every
$\mathfrak F_0$-measurable random variable $\bfU_0$ with values in $W^{1,2}(\mt,R^N)$ we have that
\begin{equation}\label{eq:4.9}
U =\mathcal{L} \left[ {\vr, \vc{q} = \vr \vu,\bfU=\bfU_{0}+\int_0^\cdot\bfu\,\dd s} \right] \in {\rm Prob}[\OTN]
\end{equation}is a solution to the martingale problem associated to \eqref{E1}--\eqref{E3} in the sense of Definition~\ref{D}.
\item Let $U$ be a solution to the martingale problem associated to \eqref{E1}--\eqref{E3} in the sense of Definition \ref{D}. Then there exists $((\mathcal O,\mf,(\mf_t)_{t\geq0},\prst),\varrho,\bfu,W)$ which is a dissipative martingale solution to \eqref{E1}--\eqref{E3} in the sense of Definition \ref{def:sol} and an $\mathfrak F_0$-measurable random variable $\bfU_0$ with values in $W^{1,2}(\mt,R^N)$ such that
\begin{equation}\label{eq:4.9a}
U =\mathcal{L} \left[ {\vr, \vc{q} = \vr \vu,\bfU=\bfU_{0}+\int_0^\cdot\bfu\,\dd s} \right] \in {\rm Prob}[\OTN].
\end{equation}
\end{enumerate}
\end{Proposition}

\begin{proof}
{\em Part 1.:} Let $\big((\mathcal O,\mf,(\mf_t)_{t\geq0},\prst),\varrho,\bfu,W)$ be a dissipative martingale solution to \eqref{E1}--\eqref{E3} in the sense of Definition \ref{def:sol} and let $\bfU_0$ be an arbitrary
 $\mathfrak F_0$-measurable random variable with values in $W^{1,2}(\mt,R^N)$.
We shall show that the probability law given by \eqref{eq:4.9} is a solution to the martingale problem associated to \eqref{E1}--\eqref{E3} in the sense of Definition \ref{D}.

The point (a) in Definition \ref{D} follows from (c), (d), (e) in Definition \ref{def:sol}, Lemma \ref{LE1} and the definition of $U$ as the pushforward measure generated by $[\vr,\bfq,\bfU]$. Similarly, we obtain that
$$
1=\mathcal{P}(\bfq=\vr\vu)=\mathcal{P}(\bfq=\vr\partial_{t}\bfU)=U(\bfxi^{2}=\xi^{1}\partial_{t}\bfxi^{3}),
$$
so (b) in Definition \ref{D} follows.
Since the total energy as well as the left hand side of the continuity equation and the renormalized equation are measurable functions on the subset of $\Omega$ where the law $U$ is supported, we deduce that the points  (c), (d), (e) in Definition \ref{D} hold.

Next, we recall that by definition of the filtration $(\mathfrak{B}_{t})_{t\geq0}$, the canonical process $\bfxi=(\xi^{1},\bfxi^{2},\bfxi^{3})$ is $(\mathfrak{B}_{t})$-adapted. Hence by Lemma \ref{LE1}, $\partial_{t}\bfxi^{3}$ is a $(\mathfrak{B}_{t})$-adapted random distribution taking values in $L^{2}_{\rm{loc}}(0,\infty;W^{1,2}(\mt,R^{N}))$.

In order to show (f) and (g) we observe that all the  expressions appearing in the definition of $\mathscr{M}(\bfphi)$ and $\mathscr{S}(\bfphi)$ are also measurable functions on the subset of $\Omega$ where $U$ is supported. Moreover, from Lemma \ref{LE1} we see that the left hand side of \eqref{N2} is a martingale with respect to the canonical filtration generated by $[\vr,\bfq,\bfU]$. This directly implies the desired martingale property of $\mathscr{M}(\bfphi)$ as follows. We consider increments $X_{t,s}=X_t-X_s$, $s\leq t$, of stochastic processes. Then we obtain for $\bfphi\in C^\infty(\mt,R^N)$
and a continuous function $h:\Omega^{[0,s]}\rightarrow[0,1]$ that
\begin{align*}
\E^U \big[h(\bfxi|_{[0,s]})\mathscr M (\bfphi)_{s,t}\big]&=
\E^\p\big[h([\varrho,\bfq,\bfU]|_{[0,s]})\mathfrak M(\bfphi)_{s,t}\big]=0,
\end{align*}
where
\begin{equation*}
\begin{split}
\mathfrak M(\bfphi)_t&=\int_{\mt}\varrho\bfu(t)\cdot\bfvarphi\dx-\int_{\mt}\varrho\bfu(0)\cdot\bfvarphi\dx-\int_0^t\int_{\mt}\varrho\bfu\otimes\bfu:\nabla\bfvarphi\dx\,\dif r\\&-\int_0^t\int_{\mt}\mathbb S(\nabla_x\bfu):\nabla_x\bfvarphi\dx\,\dif r+a\int_0^t\int_{\mt}\varrho^\gamma\cdot\diver_x\bfphi\dx\,\dif r.
\end{split}
\end{equation*}
Similarly, we obtain
\begin{align*}
&\E^U\bigg[h(\bfxi|_{[0,s]})\Big([\mathscr M (\bfphi)^2]_{s,t}-\mathscr{N}(\bfphi)_{s,t}\Big)\bigg]=\E^\p\bigg[h([\varrho,\bfq,\bfU]|_{[0,s]})\Big([\mathfrak M (\bfphi)^2]_{s,t}-\mathfrak N(\bfphi)_{s,t}\Big)\bigg]=0,
\end{align*}
where
\begin{align*}
\mathfrak N(\bfphi)_t&=\int_0^t\sum_{k=1}^\infty\left(\int_{\mt} \bfG_k(\varrho,\varrho\bfu)\cdot\bfvarphi\dx\right)^2\,\dif r,\\
\mathscr{N}(\bfphi)_t&= \int_0^t \sum_{k=1}^\infty \left( \intQ{ \vc{G}_k(\xi^1_t,\bfxi^2_t) \cdot \bfphi } \right)^2\,\dd r.
\end{align*}
As a consequence we deduce that $\mathscr M (\bfphi)$ is a $(\mathfrak{B}_{t})$-martingale with quadratic variation $\mathscr{N}(\bfphi)$.

The proof of (g) is similar to (f). In fact, there holds for any regular time $s$ and any $t\geq s$ that
\begin{align*}
\E^U \big[h(\bfxi|_{[0,s]})\mathscr S^n_{t}\big]&=
\E^\p \big[h([\vr,\bfq,\bfU]|_{[0,s]})\mathfrak S^n_{t}\big]\leq
\E^\p \big[h([\vr,\bfq,\bfU]|_{[0,s]})\mathfrak S^n_{s}\big]=\E^U \big[h(\bfxi|_{[0,s]})\mathscr S^n_{s}\big]
\end{align*}
using \eqref{N3},
where
\begin{equation*}
\begin{split}
\mathfrak S^n[\varrho,\bfu]_t&=\frac{1}{n} E_t^n + \int_0^t \left( E^{n-1}_r \intQ{ \mathbb{S}(\Grad \vu)
: \Grad \vu } \right) \,\dif r\\& - \frac{1}{2} \int_0^t \left( E^{n-1}_r \sum_{k=1}^\infty \intO{ \vr(r)^{-1} |\vc{G}_k(\vr(r),\varrho\bfu(r))|^2 }
\right) \,\dif r \\
&- \frac{n-1}{2} \int_0^\tau \left( E^{n-2}_t \sum_{k=1}^\infty \left( \intO{ \vc{G}_k(\vr(r),\varrho\bfu(r))\cdot \vc{u}(r) } \right)^2
\right) \,\dif r,\\
E_t^n&=\int_{\mt}\big[\varrho(t)|\bfu(t)|^2+P(\varrho(t))\big]\dx.
\end{split}
\end{equation*}
We have shown that $\mathscr S^n_{t}$ is an a.s. supermartingale, i.e. (g) holds.
This finishes the first part of the proof.

{\em Part 2.:} Let $U\in{\rm Prob}[\OTN]$ be a solution to the martingale problem in the sense of Definition \ref{D}.
We have to find a stochastic basis $(\mathcal O,\mf,(\mf_t)_{t\geq0},\prst)$, density $\vr$, velocity $\bfu$ and
a cylindrical
 $(\mf_t)$-Wiener process $W$ such that $((\mathcal O,\mf,(\mf_t)_{t\geq0},\prst),\vr,\bfu,W)$ is a dissipative martingale solution in the sense of Definition \ref{def:sol}.

In view of (f) in Definition \ref{D} together with the standard martingale representation theorem (see \cite{daprato}, Thm. 8.2) we infer the existence of an extended stochastic basis
\begin{equation*}
(\Omega\times\tilde\Omega,\mathfrak B\otimes{\tilde{\mathfrak B}},(\mathfrak B_{t}\otimes{\tilde{\mathfrak B}}_{t})_{t\geq0} ,U\otimes\tilde U),
\end{equation*}
and a cylindrical Wiener process
$W=\sum_{k=1}^\infty W_k e_k$ adapted to $(\mathfrak B_{t}\otimes{\tilde{\mathfrak B}}_{t})_{t\geq0}$, such that
$$\mathscr M(\bfphi)= \sum_{k=1}^\infty\int_0^t  \left( \intQ{ \vc{G}_k(\vr,\vr\vu) \cdot \bfphi } \right)\,\dd W_k,$$
where
\[
\vr(\omega, \widetilde{\omega}) := \xi^1(\omega), \quad \vu(\omega, \widetilde{\omega}) := \partial_{t}\bfxi^3(\omega),\quad\bfU(\omega,\tilde\omega):=\bfxi^{3}(\omega).
\]
Choosing for $(\mathcal O,\mf,(\mf_t)_{t\geq0},\prst)$ the above extended probability space with the corresponding augmented filtration, then
$
 \big((\mathcal O,\mf,(\mf_t)_{t\geq0},\prst),\vr,\vu,W\big)
$
is a dissipative martingale solution solution to (\ref{E1})--(\ref{E3})
in the sense of Definition \ref{def:sol}. Furthermore, it holds
\begin{equation*}
U =\mathcal{L} \left[ {\vr, \vc{q} = \vr \vu,\bfU=\bfU_{0}+\int_0^\cdot\bfu\,\dd s} \right] \in {\rm Prob}[\OTN],
\end{equation*}
where by definition $\bfU_{0}(\omega,\tilde\omega)=\bfxi^{3}_{0}(\omega).$
\end{proof}

In other words, to every dissipative martingale solution we may associate infinitely many solutions to the martingale problem, which are parametrized by the initial condition $\bfU_{0}$, but whose marginals corresponding to $(\xi^{1},\bfxi^{2})$ coincide. On the other hand, as it can be seen from the proof of {\em Part 2.} of the proof of Proposition \ref{prop:def=D} that having a solution to the martingale problem already determines the initial condition $\bfU_{0}$ used in \eqref{eq:4.9a}.

\section{Main result \& proof}
\label{s:main}

In this section we present our main result which is the existence of an almost sure Markov selection to the compressible Navier--Stokes \eqref{E1}--\eqref{E3}. In the following, if $y\in Y$ is an admissible initial condition, we denote by $U_{y}$ a solution to the martingale problem associated to \eqref{E1}--\eqref{E3} starting from $y$ at time $t=0$. That is, the marginal of $U_{{y}}$ at $t=0$ is $\delta_{y}$.

\begin{Theorem}\label{thm:main}
Let
\[
p(\vr) = a \vr^\gamma, \ a > 0, \ \gamma > \frac{N}{2}.
\]
 Let the diffusion coefficients $\mathbb{G} = (\mathbf{G}_{k})_{k\in\N}$ be continuously differentiable
satisfying \eqref{growth1'} and \eqref{growth2'}. Then there exists a family $\{U_{y}\}_{y\in Y}$ of solutions to the martingale problem associated to \eqref{E1}--\eqref{E2} in the sense of Definition \ref{D} with the  Markov property
(as defined in Definition \ref{def:2.4}).
\end{Theorem}

For each $y=(y^1,\bfy^2,\bfy^3)\in Y$ we denote by $\CC_{\rm{NS}}(y)$  the set of probability laws
$U^y\in\PO$ solving the martingale problem associated to \eqref{E1}--\eqref{E3} with the initial law $\delta_{y}$.
In order to prove Theorem \ref{thm:main} we aim to apply the abstract result from Theorem \ref{thm:2.8} to the family $\{\CC_{\rm{NS}}(y)\}_{y\in Y}$ of solutions to the martingale problem.

\begin{Proposition}\label{prop:convex}
Let $y=(y^1,\bfy^2,\bfy^3)\in Y$. Then $\CC_{\mathrm{NS}}(y)$ is nonempty and convex. Moreover, for every $U\in \CC_{\mathrm{NS}}(y)$, the marginal at every time $t\in(0,\infty)$ is supported on $Y$.
\end{Proposition}

\begin{proof}
If $y\in Y$ then the assumptions of Theorem
\ref{thm:existence} are satisfied for the initial law $\Lambda=\delta_{(y^1,\bfy^2)}$ and existence of a dissipative martingale solution $\big((\mathcal O,\mf,(\mf_t)_{t\geq0},\prst),\varrho,\bfu,W)$ in the sense of Definition \ref{def:sol} with initial law $\delta_{(y^1,\bfy^2)}$ follows.
In view of Proposition \ref{prop:def=D} we therefore deduce that for each $y\in Y$ the set $\CC_{\mathrm{NS}}(y)$ is not empty.

In order to check the convexity, let $U_{{1}},U_{{2}}\in \CC_{\mathrm{NS}}(y)$ and let $U=\lambda U_{{1}}+(1-\lambda)U_{{2}}$ for some $\lambda\in (0,1)$. Since all the points in Definition \ref{D} involve integration with respect to the measure $U$ of measurable functions on the subset of $\Omega$ where the measure is supported and we work with the canonical process $\bfxi$, the convexity follows immediately.

Finally, as a consequence of the energy inequality, see in particular Remark \ref{rem:moments}, the marginal of $U\in \CC_{\mathrm{NS}}(y)$ at every $t\in(0,\infty)$ is supported on $Y$.
\end{proof}

In order to apply Theorem \ref{thm:2.8} it remains to show
 compactness as well as the disintegration and reconstruction property of the family $\{\CC_{\mathrm{NS}}(y)\}_{y\in Y}$. We are going to do this in the following subsections (see Proposition \ref{Pnn1}--\ref{prop:4.5}). Theorem \ref{thm:main}
follows then from Theorem~\ref{thm:2.8}.

\subsection{Compactness}

The following proposition yields compactness of $\CC_{\mathrm{NS}}(y)$ (choosing $y_{m}$ constant) as well as measurability of the map $y\mapsto \CC_{\mathrm{NS}}(y)$ (using \cite[Thm. 12.1.8]{StVa} for the metric space $(Y,d_H)$).

\begin{Proposition} \label{Pnn1}
Let $(y_m=(\varrho_{m,0},\bfq_{m,0},\bfU_{m,0}))_{m\in\N}\subset Y$ be a sequence converging to some $y=(\varrho_{0},\bfq_{0},\bfU_{0})\in Y$ with respect to the metric $d_H$ given in \eqref{eq:dy}.
Let $U_m\in\mathcal C_{\mathrm{NS}}(y_m)$, $m\in\N$. Then $(U_m)_{m\in\N}$ has a subsequence that converges to
some $U\in\CC_{\mathrm{NS}}(y)$ weakly in $\PO$.
\end{Proposition}

\begin{proof}
On account of Proposition \ref{prop:def=D} there is a sequence $((\mathcal O^{m},\mf^{m},(\mf^m_t)_{t\geq0},\prst),\vr_m,\bfu_m, W_{m})$ of dissipative martingale solutions in the sense of Definition \ref{def:sol} and an $\mf^m_0$-measureable random variable $\bfU_{m,0}$ with values in $W^{1,2}(\mt,R^N)$ such that
\[
U_m = \mathcal L \left[\vr_m, \vc{q}_m = \vr_m \vu_m,\bfU_m=\bfU_{m,0}+\int_0^\cdot\vu_m\,\dd s\right].
\]
Choosing $s=0$ in \eqref{N3} we have for any $0<t\leq T$
\begin{align*}
&\frac{1}{n}  \left[ \intQ{ \left( \frac{1}{2} \vr_m |\vu_m|^2 + P(\vr_m) \right) } \right]^n
\\&+ \int_0^t \left( \left[ \intQ{ \left( \frac{1}{2} \vr_m |\vu_m|^2 + P(\vr_m) \right) } \right]^{ n - 1}
\intQ{ \mathbb{S}(\Grad \vu_m) : \Grad \vu_m } \right) \dtau
\\ & \leq \frac{1}{2} \int_0^t \left( \left[ \intQ{ \left( \frac{1}{2} \vr_m |\vu_m|^2 + P(\vr_m) \right) } \right]^{ n - 1}
\sum_{k=1}^\infty \intQ{ \vr_m^{-1} |\vc{G}_k(\vr_m,\vr_m\bfu_m)|^2 } \right) \dtau  \\
&+ \frac{n-1}{2} \int_0^t \left( \left[ \intQ{ \left( \frac{1}{2} \vr_m |\vu_m|^2 + P(\vr_m) \right) } \right]^{ n - 2 }
\sum_{k=1}^\infty \left( \intQ{  \vc{G}_k (\vr_m,\vr_m\bfu_m)\cdot \vu_m } \right)^2 \right) \dtau
\\
&+ \int_0^t \left( \left[ \intQ{ \left( \frac{1}{2} \vr_m |\vu_m|^2 + P(\vr_m) \right) } \right]^{ n - 1}
\sum_{k = 1}^\infty \intQ{ \vc{G}_k(\vr_m,\vr_m\bfu_m) \cdot \vu_m } \right) {\rm d}W_{m,k}\\
&+\frac{1}{n} \left[ \intQ{ \left( \frac{|\bfq_{m,0}|^2}{2\varrho_{m,0}} + P(\vr_{m,0}) \right) } \right]^n.
\end{align*}
By assumption on the initial data and the definition of the metric $d_Y$ in \eqref{eq:dy} the last term stays bounded uniformly in $m$.
All the other terms on the right-hand side, which we denote by $(I),(II)$
and $(III)$, need to be estimated. By \eqref{growth1'} we have
\begin{align*}
(I)&\leq \,c\, \int_0^T \left( \left[ \intQ{ \left( \frac{1}{2} \vr_m |\vu_m|^2 + P(\vr_m) \right) } \right]^{ n - 1}
\intQ{\big(\vr_m+\vr_m|\bfu_m|^2\big)} \right) \dt\\
&\leq\,c(T)+c\,  \int_0^T  \left[ \intQ{ \left( \frac{1}{2} \vr_m |\vu_m|^2 + P(\vr_m) \right) } \right]^{ n }\dt.
\end{align*}
Similarly, we obtain
\begin{align*}
(II)&\leq\,c\,\int_0^T  \left[ \intQ{ \left( \frac{1}{2} \vr_m |\vu_m|^2 + P(\vr_m) \right) } \right]^{ n - 2 }\times\\
&\times\sum_{k=1}^\infty\intQ{\varrho_m^{-1}|\bfG_k(\varrho_m,\varrho_m\bfu_m)|^2} \intQ{\varrho_m|\bfu_m|^2} \dt\\
&\leq\,c(T)+c\,\int_0^T  \left[ \intQ{ \left( \frac{1}{2} \vr_m |\vu_m|^2 + P(\vr_m) \right) } \right]^{ n }\dt.
\end{align*}
Finally, the Burkholder-Davis-Gundy inequality yields for $r>0$ arbitrary
\begin{align*}
\E&\bigg[\sup_{0\leq t\leq T}(III)\bigg]^r\\&\leq\,c\,\E\bigg[\int_0^T \left[ \int_{\Q} \left( \frac{1}{2} \vr_m |\vu_m|^2 + P(\vr_m) \right) \dx \right]^{ 2n - 2}
\sum_{k = 1}^\infty \left(\int_{\Q} \vc{G}_k(\vr_m,\vr_m\bfu_m) \cdot \vu_m \dx \right)^2\dt\bigg]^{\frac{r}{2}}\\
&\leq\,c\,\E\bigg[\int_0^T \left[ \int_{\Q} \left( \frac{1}{2} \vr_m |\vu_m|^2 + P(\vr_m) \right) \dx \right]^{ 2n - 2}\times\\& \times
\sum_{k = 1}^\infty \int_{\Q}\varrho_m^{-1}|\bfG_k(\varrho_m,\varrho_m\bfu_m)|^2\dx \int_{\Q}\varrho_m|\bfu_m|^2\dx \dt\bigg]^{\frac{r}{2}}\\
&\leq\,c(T)+c\, \bigg[\int_0^T\bigg[\int_{\Q} \left( \frac{1}{2} \vr_m |\vu_m|^2 + P(\vr_m) \right) \dx \bigg]^{ 2n }\dt\bigg]^{\frac{r}{2}}\\
&\leq\,c(T)+\frac{1}{2}\, \bigg[\sup_{0\leq t\leq T}\bigg[\int_{\Q} \left( \frac{1}{2} \vr_m |\vu_m|^2 + P(\vr_m) \right) \dx \bigg]^{ n }\dt\bigg]^r\\&+c\, \E\bigg[\int_0^T\bigg[\int_{\Q} \left( \frac{1}{2} \vr_m |\vu_m|^2 + P(\vr_m) \right) \dx \bigg]^{ n }\dt\bigg]^r.
\end{align*}
By Gronwall's lemma we get
\begin{align}\label{apriori1}
& \E\bigg[\sup_{0\leq t\leq T}\left[ \intQ{ \left( \frac{1}{2} \vr_m |\vu_m|^2 + P(\vr_m) \right) } \right]^n\bigg]^r
\\&+ \E\bigg[\int_0^T \left( \left[ \intQ{ \left( \frac{1}{2} \vr_m |\vu_m|^2 + P(\vr_m) \right) } \right]^{ n - 1}\nonumber
\intQ{ \mathbb{S}(\Grad \vu_m) : \Grad \vu_m } \right) \dtau\bigg]^r\\&\leq\,c(n,r,T)\nonumber
\end{align}
uniformly in $m$ for all $T>0$ and all $n,r\in\N$. As in \cite[Section 4.5.2]{BFHbook} we can infer from \eqref{apriori1} the following uniform pressure estimate
\begin{equation*}
\expe{ \left|\int_0^T  \intQ{  p(\vr_m) \vr^\beta_m } \dt \right|^r } \aleq c(r,\Ov{\vr},T)
\end{equation*}
for a certain $\beta > 0$. Here $\overline\vr$ is given by
\begin{align*}
\overline\vr&=\sup_{m\in\N} \intQ{\vr_{m,0}}.
\end{align*}
As in \cite[Prop. 4.5.4.]{BFHbook} we can use momentum and continuity equation to gain  information about the regularity in time:
There exist $\kappa>0$ such that
\begin{equation}\label{apriori3}
\stred\big\|(\vr_m,\vr_m\bfu_m)\|_{C^\kappa([0,T];W^{-k,2}(\mt))}\leq C(T)
\end{equation}
for all $T>0$.
Combining \eqref{apriori1}--\eqref{apriori3} we can show that the
family of joint laws
\[
\left\{\mathcal{L} \left[ \vr_m,
\vr_m \vu_m,\bfu_m, \bfU_m, W_{m}, \delta_{[\vr_m, \vu_m,\nabla\vu_m]}\right] ;\,m\in\N\right\}
\]
is tight on
\[
\mathcal{X} =  \mathcal{X}_{\vr} \times \mathcal{X}_{\vr \vu} \times \mathcal{X}_{\bfu} \times \mathcal{X}_{\bfU}
\times \mathcal{X}_{W} \times \mathcal{X}_{\nu},
\]
where
\begin{align*}
\mathcal{X}_{\vr} &= \big(L^{\gamma + \beta}_{\mathrm{loc}}(0,\infty;L^{\gamma + \beta}(\Q)),w\big) \cap C_{\mathrm{loc}}\big([0,\infty);\big( L^\gamma(\Q),w\big)\big)\cap
C_{\mathrm{loc}}([0,\infty); W^{-k,2}(\Q)),\\
\mathcal{X}_{\vr \vu} &= C_{\mathrm{loc}}\big([0,\infty); \big(L^{\frac{2\gamma}{\gamma + 1}}(\Q;R^N),w\big)\big)\cap
C_{\mathrm{loc}}([0,\infty); W^{-k,2}(\Q;R^N)),\\
\mathcal{X}_{\bfu} &= \left(L^{2}_{\mathrm{loc}}(0,\infty; W^{1,2}(\Q,R^N)),w\right),\\
\mathcal{X}_{\bfU} &= \left(W^{1,2}_{\mathrm{loc}}(0,\infty; W^{1,2}(\Q,R^N)),w\right),\\
\mathcal{X}_{W} &= C_{\mathrm{loc}}([0,T); \mathfrak{U}_0)\\
\mathcal{X}_{\nu} &= (L^\infty_{\mathrm{loc}}((0,\infty) \times \Q; {\rm Prob}( \R^{13})),w^*).
\end{align*}
Note that we also include the Young measure $\nu_m$ associated to $[\vr_{m},\vu_m,\nabla\vu_m]$, that is the weak-$*$ measurable mapping
$$\nu_m:[0,T]\times\mt\to {\rm Prob}\big(\R\times\R^3\times\R^{3\times 3}\big)\simeq {\rm Prob}\big(\R^{13}),\quad
\nu_{m,t,x}(\cdot)=\delta_{[\vr_m,\vu_m,\nabla\vu_m](t,x)}(\cdot).
$$

In particular, we have shown that $\{U_m;\,m\in\N\}$ is tight. By Prokhorov's theorem there is a subsequence converging weakly to some
$U\in\PO$. It remains to show that $U\in \mathcal C_{\mathrm{NS}}(y)$. Following \cite[Prop. 4.5.5.]{BFHbook} we obtain by the Jakubowski--Skorokhod theorem \cite{jakubow} the existence of a complete probability space $(\tilde\Omega,\tilde\mf,\tilde\prst)$ with $\mathcal{X}$-valued  random variables $\left[\tilde\vr_m, \tilde{\bfq}_m,\tilde\bfu_m, \tilde\bfU_m, \tilde W_m, \tilde\nu_{m} \right]$, $m\in\N$, as well as $\left[\tilde\vr, \tilde{\bfq},\tilde\bfu, \tilde\bfU, \tilde W, \tilde\nu \right]$ such that (up to a subsequence)
\begin{enumerate}
 \item  $\mathcal L\left[\tilde\vr_m, \tilde{\bfq}_m, \tilde\bfu_m,\tilde\bfU_m,\tilde W_m,  \tilde\nu_{m} \right]$ and $\mathcal L\left[ {\vr}_m, {\vr_m \vu_m},\bfu_m, {\bfU}_m,  {W}_{m}, {\delta}_{[{\vr_m}, {\vu_m},\nabla\vu_m]}\right]$ coincide for all $m\in\N$, in particular, we have  $\mathcal L[{\tilde\vr}_m, {\tilde\vr_m \tilde\vu_m},\tilde\bfU_{m}]=U_m$;
\item the law of $\left[\tilde\vr,\tilde\bfq,\tilde\bfu, \tilde{\bfU}, \tilde{W} , \tilde{\nu}\right]$ on $\mathcal{X}$  is a Radon measure,
 \item $\left[\tilde{\vr}_m, \tilde{\vr}_m \tilde\vu_m, \tilde\bfu_m,\tilde{\bfU}_m,  \tilde{W}_m , \tilde{\nu}_{m}\right]$ converges $\tilde{\mathbb P}$-a.s. to $\left[\tilde\vr, \tilde{\bfq}, \tilde\bfu,\tilde\bfU, \tilde W, \nu \right]$ in the topology of $\mathcal{X}$;

\item for any Carath\'eodory function $H=H(t,x,\rho,\bfv,\bfV)$ where $(t,x)\in (0,T)\times\mt,$ $(\rho,\bfv,\bfV)\in \R^{13}$, satisfying for some $q>0$ the growth condition
$$
|H(t,x,\rho,\bfv,\bfV)|\lesssim 1 +|\rho|^{q_1}+|\bfv|^{q_2}+|\bfV|^{q_2},
$$
uniformly in $(t,x)$,
denote $\Ov{H(\tilde{\vr}, \tilde{\vu},\nabla\tilde{\vu}) }(t,x)=\langle \tilde\nu_{t,x},H\rangle$. Then it holds true that
\begin{equation*}
\begin{split}
H(\tilde{\vr}_m, \tilde{\vu}_m,\nabla\tilde\vu_m) &\rightharpoonup \Ov{H(\tilde{\vr}, \tilde{\vu}, \nabla\tilde\vu) } \ \mbox{ in }\ L^r((0,T) \times \mt)\\
\ \mbox{ for all }\ & 1<r\leq \frac{\gamma+\beta}{q_1}\wedge \frac{2}{q_2},
\end{split}
\end{equation*}
as $m \to \infty$  $\tilde{\mathbb{P}}$-a.s.
\end{enumerate}

Finally,  it remains to show that $U$ which is the probability law of $[\tilde\vr,\tilde\bfq,\tilde\bfU]$ solves the martingale problem associated to \eqref{E1}--\eqref{E3}. First observe that we have $\tilde\bfu=\partial_{t}\tilde\bfU$ and $\tilde\bfq=\tilde\vr\tilde\vu$ $\tilde\p$-a.s.
Next, we observe that $[\tilde\vr,\tilde\vu]$ is a dissipative martingale solution
to \eqref{E1}--\eqref{E3} in the sense of Definition \ref{def:sol}. This can be shown exactly as in \cite[Sect. 4.5.1.]{BFHbook}. The only difference is the improved energy inequality \eqref{N3}. However, the main difficulty is still the passage to the limit in the stochastic integral, which can be shown along the lines of \cite[Prop. 4.4.13.]{BFHbook}. The rest follows from the improved moment estimates given in \eqref{apriori1}.

In addition, from the convergence of the deterministic initial conditions $\bfU_{m,0}\to\bfU_{0}$ it follows that $\tilde\bfU_{0}=\bfU_{0}$. To summarize, we proved that
$$
U=\mathcal{L}\left[\tilde\vr,\tilde\vr\tilde\vu,\tilde\bfU=\bfU_{0}+\int_{0}^{\cdot}\tilde\vu\,\dd s\right],
$$
where $[\tilde\vr,\tilde\vu]$ is a dissipative martingale solution to \eqref{E1}--\eqref{E3}. According to Proposition \ref{prop:def=D}, $U$ is a solution to the martingale problem with the initial law given by $\delta_{(\vr_{0},\bfq_{0},\bfU_{0})}$. Hence $U\in\CC_{\rm{NS}}(y)$ and the proof is complete.
\end{proof}

\subsection{Disintegration property}
In this subsection we prove that the family $\{\CC_{\mathrm{NS}}(y)\}_{y\in Y}$ is stable with respect to disintegration.
\begin{Proposition}\label{prop:4.4}
The family $\{\CC_{\mathrm{NS}}(y)\}_{y\in Y}$ has the disintegration property of Definition~\ref{def:2.5}.
\end{Proposition}

\begin{proof}
Let $y\in Y$ and $U \in \CC_{\mathrm{NS}}(y)$. Further let $T \geq 0$ be a regular time (i.e. a time at which the energy inequality in the sense of Definition \ref{D} (e) holds). In accordance with Theorem~\ref{T1},
there is a family of probability measures,
\[
\OTN\ni\tilde{\omega}\mapsto U|^{\tilde{\omega}}_{\mathfrak{B}_T} \in {\rm Prob}[\OTp]
\]
such that
\begin{equation} \label{De2}
\begin{split}
&\omega (T) = \tilde{\omega}(T), \ U|^{\tilde{\omega}}_{\mathfrak{B}_T}\mbox{-a.s.},\quad
U \left( \omega|_{[0,T]} \in {A}, \ \omega|_{[T, \infty)}
\in {B} \right) = \int_{\tilde{\omega} \in {A}}
U|^{\tilde{\omega}}_{\mathfrak{B}_T} \left( {B} \right) \ {\rm d}U,
\end{split}
\end{equation}
for any Borel sets ${B} \subset \OTp$, ${A} \subset \OTm$.
Our goal is to prove that
\begin{align*}
\mathcal S_{-T} U|_{\mathfrak B_\tau}^{\tilde \omega}\in \mathcal C(\omega(T))\ \mbox{for} \ {\tilde\omega} \in \OTN \ U-\mbox{a.s.}
\end{align*}
We aim at finding an $U|^{\tilde\omega}_{\mathfrak{B}_{T}}$-nullset $N$ outside of which points (a)--(f) from Definition \ref{D} hold for $U|^{\tilde\omega}_{\mathfrak{B}_{T}}$. In fact, we will relate nullsets $N_a,\dots, N_f$ to each of the points (a)--(f) and set $N=N_a\cup\dots\cup N_f$. The crucial part hereby is (b), the rest is similar to \cite[Lemma 4.4]{FlaRom}.
However, we still give the details for the convenience of the reader.\\
(a) Setting
 \begin{align*}
S_T&=\Big\{\omega\in\OTN:\,\,\omega|_{[0,T]}\in C\big([0, T]; \big(L^{\gamma}(\mt),w\big)\big)\times C\big(0,T;\big(L^{\frac{2\gamma}{\gamma+1}}(\mt;R^N),w\big)\big)\times\\&\qquad\qquad\qquad\qquad\quad\times W^{1,2}(0,T;W^{1,2}(\mt;R^N))\Big\},\\
S^T&=\Big\{\omega\in\OTN:\,\,\omega|_{[T,\infty)}\in C_{\mathrm{loc}}\big([T,\infty);\big(L^\gamma(\mt),\big)\big)\times C_{\mathrm{loc}}\big([T,\infty);\big(L^{\frac{2\gamma}{\gamma+1}}(\mt;R^N),w\big)\big)\times\\&\qquad\qquad\qquad\qquad\quad\times W^{1,2}_{\mathrm{loc}}(T,\infty;W^{1,2}(\mt;R^N))\Big\},
\end{align*}
we obtain by (a) for $U$ that
$$1=U(S_T\cap S^T)=\int_{S_T}U|^{\tilde\omega}_{\mathfrak{B}_{T}}(S^T)\,\dd U(\tilde\omega).$$
Consequently, there holds $U|^{\tilde\omega}_{\mathfrak{B}_{T}}(S^T)=1$
for $U$-a.a. $\tilde\omega$. The remaining $\tilde\omega\in\Omega$ are contained in a nullset $N_a$.\\
(b) Due to \eqref{De2} and (a) for $U$ we obtain
\begin{align*}
1&=U(\bfxi^2=\xi^1\ups)=U\big(\bfxi^2|_{[0,T]}=(\xi^1\ups)|_{[0,T]},\bfxi^2|_{[T,\infty)}=(\xi^1\ups)|_{[T,\infty)}\big)\\
&=\int_{\{\bfxi^2|_{[0,T]}=(\xi^1\ups)|_{[0,T]}\}}U|^{\tilde\omega}_{\mathfrak{B}_{T}}\big(\bfxi^2|_{[T,\infty)}=(\xi^1\ups)|_{[T,\infty)}\big)\dd U(\tilde\omega),
\end{align*}
which implies that
$$
U|^{\tilde\omega}_{\mathfrak{B}_{T}}\big(\bfxi^2|_{[T,\infty)}=(\xi^1\ups)|_{[T,\infty)}\big)=1\quad U\text{-a.e. }\tilde\omega\in\Omega.
$$
So, we have $\bfxi^2|_{[T,\infty)}=(\xi^1\ups)|_{[T,\infty)}$ $U|^{\tilde\omega}_{\mathfrak{B}_{T}}$-a.s. and obtain the nullset $N_b$.\\
(c) Recalling $\mathcal E_{t} = \int_{\mt}\left[ \frac{1}{2} \frac{|\bfxi_{t}^2|^2}{\xi_{t}^1} +P(\xi_{t}^1) \right]\dx $ we consider the sets
 \begin{align*}
\mathfrak E_T&=\Big\{\omega\in\OTN:\,\,\mathcal E|_{[0,T]}\in L^\infty_{\mathrm{loc}}(0,T)\Big\},\\
\mathfrak E^T&=\Big\{\omega\in\OTN:\,\,\mathcal E|_{[T,\infty)}\in L^\infty_{\mathrm{loc}}(T,\infty)\Big\}.
\end{align*}
As (c)  holds for $U$ we can argue as in the proof of (a) (replacing $S_T$ and $S^T$ by $\mathfrak E_T$ and $\mathfrak E^T$ respectively)
to conclude that $U|^{\tilde\omega}_{\mathfrak{B}_{T}}(\mathfrak E^T)=1$
for $U$-a.a. $\tilde\omega$. This gives us the nullset $N_c^n$.
(d)  Let $(\psi_n)_{n\in\N}$ be a dense subset of $W^{k,2}(\mt)$. To each $n\in\N$ we will relate an $U$-nullset $N_c^n$ and set $N_d=\cup_{n\in\N}N_d^n$. Let us fix some $n\in\N$. We split the continuity equation into two part, namely
\begin{align}
\left[ \intQ{ \xi_{t}^1 \psi_n } \right]_{t = 0}^{t = \tau} +
\int_0^\tau \intQ{ \bfxi_{t}^2 \cdot \Grad \psi_n  } \dt = 0 \quad \forall0\leq\tau\leq T\label{C1},\\
\left[ \intQ{ \xi_{t}^1 \psi_n } \right]_{t = T}^{t = \tau} +
\int_0^\tau \intQ{ \bfxi_{t}^2 \cdot \Grad \psi_n  } \dt = 0 \quad \forall T\leq\tau<\infty. \label{C2}
\end{align}
Now we consider the sets
 \begin{align*}
R_T&=\Big\{\omega\in\OTN:\,\,\omega|_{[0,T]}\text{ satisfies \eqref{C1}}\Big\},\\
R^T&=\Big\{\omega\in\OTN:\,\,\omega|_{[T,\infty)}\text{ satisfies \eqref{C2}}\Big\}.
\end{align*}
As (d)  holds for $U$ we can argue again
as in the proof of (a) and (c) to obtain
the nullset $N_d^n$.
 \\
(e) follows by exactly the same reasoning as in (d).\\
(f) Let $(\bfphi_n)_{n\in\N}$ be a dense subset of $W^{k,2}(\mt;R^N)$. As in (d) we will set $N_f=\cup_{n\in\N}N_f^n$ where for each $n\in\N$ we will obtain an $U$-nullset $N_f^n$. Since (f) holds for $U$ we know that $(\mathscr M_t(\bfphi_n))_{t\geq0}$ is a $((\mB_t)_{t\geq0},U)$-square integrable martingale with quadratic variation
\[
\mathscr{N}(\bfphi_{n})_\tau=\frac{1}{2} \int_0^\tau \sum_{k=1}^\infty \left( \intQ{ \vc{G}_k(\xi_{t}^1,\bfxi_{t}^2) \cdot \bfphi_n } \right)^2\dt.
\]
On account of Proposition \ref{prop:B1} we obtain for $U$-a.a. $\tilde\omega$ that $(\mathscr M_t(\bfphi_n))_{t\geq T}$ is a $((\mB_t)_{{t\geq T}},U|^{\tilde\omega}_{\mathfrak{B}_{T}})_{t\geq T}$-square integrable martingale with quadratic variation $(\mathscr{N}(\bfphi_{n}))_{{t\geq T}}$.\\
(g) We recall that
$\mathcal{E}_t=\int_{\mt}\big[\tfrac{|\bfxi^{2}_{t}|^2}{2\xi^{1}_{t}}+P(\xi^{1}_{t})\big]\dx$
and decompose the process $\mathscr S_\tau^n$ as $\mathscr S_\tau^n=\alpha_\tau^n-\beta_\tau^n$, where
\[
\begin{split}
\alpha_\tau^n&= \frac{1}{n} \mathcal{E}_\tau^n + \int_0^\tau \left( \mathcal{E}^{n-1}_t \intQ{ \mathbb{S}(\Grad \ups_{t})
: \Grad \ups_{t} } \right) \dt,\\
\beta_\tau^n&= \frac{1}{2} \int_0^\tau \left( \mathcal{E}^{n-1}_t \sum_{k=1}^\infty \intO{  \frac{|\vc{G}_k(\xi_{t}^1,\bfxi_{t}^2)|^2}{\xi_t^1} }
\right) \dt \\
&+ \frac{n-1}{2} \int_0^\tau \left( \mathcal{E}^{n-2}_t \sum_{k=1}^\infty \left( \intO{ \vc{G}_k(\xi_{t}^1,\bfxi_{t}^2)\cdot \ups_{t} } \right)^2
\right) \dt .
\end{split}
\]
By (g) for $U$ we know that $ (\mathscr S_t^n)_{t\geq T}$ is an almost sure $((\mB_t)_{t\geq T},U)$-supermartingale and we can show iteratively that $\E^U[\alpha_\tau^n]<\infty$ and $\E^U[\beta_\tau^n]<\infty$ for all $n\in \N$.
We note that $\alpha^{n}_\tau$ is lower semi-continuous, and $\beta_\tau$ non-decreasing such that $\mathscr S^{n}_\tau$ is left lower semi-continuous. Hence,
Proposition \ref{prop:B2} is applicable and yields for $U$-a.a. $\tilde\omega$ that $(\mathscr S_t^n)_{t\geq T}$ is an almost sure $((\mB_t)_{t\geq T},U|^{\tilde\omega}_{\mathfrak{B}_{T}})$-supermartingale. This gives an $U$-nullset $N_g^n$ and we set $N_g=\cup_{n\in\N}N_g^n$.
\end{proof}

\subsection{Reconstruction}
In this subsection we prove that the family $\{\CC_{\mathrm{NS}}(y)\}_{y\in Y}$ is stable with respect to reconstruction.
\begin{Proposition}\label{prop:4.5}
The family $\{\CC_{\mathrm{NS}}(y)\}_{y\in Y}$ has the reconstruction property of Definition~\ref{def:2.5}.
\end{Proposition}

\begin{proof}
Let $y\in Y$ and $U \in \CC_{\mathrm{NS}}(y)$. Further let $T \geq 0$ be a regular time (i.e. a time at which the energy inequality in the sense of Definition \ref{D} (g) holds). We shall prove the following:
Let $Q_{\omega}:\Omega\to\mathrm{Prob}(\Omega^{[T,\infty)})$ be a $\mathfrak{B}_{T}$-measurable map such that there is $N\in\mathfrak{B}_{T}$ with $U(N)=1$ and for all $\omega\notin N$ it holds
$$
\omega(T)\in Y\quad \text{and} \quad \mathcal{S}_{-T}Q_{\omega}\in \CC_{\mathrm{NS}}(\omega(T));
$$
then $U\otimes_{T}Q\in \CC_{\mathrm{NS}}(y)$. So, we have to verify points (a)--(f) from Definition \ref{D} for $U\otimes_{T}Q$.
As in the proof
of Proposition \ref{prop:4.4} the crucial point is (b) and the rest follows along the lines of \cite[Lemma 4.4]{FlaRom}.\\
(a) As (a) holds for $Q_\omega$ we have $Q_\omega[S^T]=1$ (using the notation from the proof of Proposition \ref{prop:4.4}) such that
$$U\otimes_{T} Q[S_T\cap S^T]=\int_{S_T}Q_\omega[S^T]\,\dd P(\omega)=1.$$
Finally, with probability 1 we have that $\bfxi^3$ is continuous and hence weakly differentiable at time $T$ such that (a) follows.\\
(b) Due to the definition of $U\otimes_{T} Q$, the measure $Q_{\omega}$ is a regular conditional probability distribution of $U\otimes_{T}Q$ on $\mathfrak{B}_{T}$. Hence, it holds
$$
(U\otimes_{T}Q)(A\cap B)=\int_{A}Q_{\omega}(B)\dd U(\omega)
$$
for every two Borel sets $A\subset\Omega^{[0,T]}$ and $B\subset\Omega^{[T,\infty)}$.
Accordingly,
\begin{align*}
(U\otimes_{T}Q)(\bfxi^2=\xi^1\ups)&=(U\otimes_{T}Q)\Big(\bfxi^2|_{[0,T]}=(\xi^1\ups)|_{[0,T]},\bfxi^2|_{[T,\infty)}=(\xi^1\ups)|_{[T,\infty)}\Big)\\
&=\int_{\{\bfxi^2|_{[0,T]}=(\xi^1\ups)|_{[0,T]}\}}Q_{\omega}\Big(\bfxi^2|_{[T,\infty)}=(\xi^1\ups)|_{[T,\infty)}\Big)\dd U(\omega)\\
&=\int_{\{\bfxi^2|_{[0,T]}=(\xi^1\ups)|_{[0,T]}\}}\dd U(\omega)=1.
\end{align*}
This completes the proof of (b).\\
(c)  Using the notation from the proof of Proposition \ref{prop:4.4} we have
$$U\otimes_{T} Q[\mathfrak E_T\cap \mathfrak E^T]=\int_{\mathfrak E_T}Q_\omega[\mathfrak E^T]\,\dd U(\omega)=1$$
due to $Q_\omega(\mathfrak E^T)=1$ by (c) for $Q_\omega$. Consequently, we have $\mathcal E\in L^\infty_{\mathrm{loc}}(0,\infty)$ $U\otimes_{T} Q-$a.s.
\\
(d) Using the notation from the proof of Proposition \ref{prop:4.4} we have
$$U\otimes_{T} Q[R_T\cap R^T]=\int_{R_T}Q_\omega[R^T]\,\dd U(\omega)=1$$
due to $Q_\omega(R^T)=1$ by (d) for $Q_\omega$.\\
(e) follows by exactly the same reasoning as in (d).\\
(f) As (d) holds for $Q_\omega$ we know that $(\mathscr M_t(\bfphi))_{t\geq T}$ is a $((\mB_t)_{t\geq T},Q_\omega)$-square integrable martingale for all $\bfphi\in C^1(\mt;R^N)$. By Proposition \ref{prop:B1} we obtain that $(\mathscr M_t(\bfphi))_{t\geq T}$ is a $((\mB_t)_{t\geq T},U\otimes_{T} Q)_{t\geq T}$-square integrable martingale as well. Since $U$ and $U\otimes_T Q$ coincide on $\BTm$ and $(\mathscr M_t(\bfphi))_{0\leq t\leq T}$ is a
 $((\mB_t)_{0\leq t\leq T},U)$-martingale (as $U$ satisfies (f)) we conclude that $(\mathscr M_t(\bfphi))_{t\geq 0}$ is a $((\mB_t)_{t\geq 0},U\otimes_{T} Q)$ is a martingale.\\
(g) We use again the notation from the proof of Proposition \ref{prop:4.4}
and recall that $\mathscr S^n_\tau$ is left lower semi-continuous.
As (g) holds for $Q_\omega$ we know that $(\mathscr S_t^n)_{t\geq T}$ is an almost sure $((\mB_t)_{t\geq T},Q_\omega)$-supermartingale. By Proposition \ref{prop:B2} we obtain that $(\mathscr S^n_t)_{t\geq T}$ is an almost sure $((\mB_t)_{t\geq T},U\otimes_{T} Q)$-supermartingale as well. Since $U$ and $U\otimes_T Q$ coincide on $\BTm$ and $(\mathscr S_t^n)_{0\leq t\leq T}$ is an almost sure
 $((\mB_t)_{0\leq t\leq T},U)$-supermartingale (as $U$ satisfies (g)) we conclude that $(\mathscr S_t^n)_{t\geq0}$ is an almost sure $((\mB_t)_{t\geq 0},U\otimes_{T} Q)$-supermartingale.
\end{proof}

\end{document}